\documentclass[]{siamart220329}
\usepackage{amsmath}
\usepackage{amssymb}
\usepackage{graphicx}

\usepackage{tikz}
\usetikzlibrary{external}
\usetikzlibrary{arrows.meta,positioning}
\usepackage{pgfplots}
\usepackage[ComplainAboutUnexternalized]{externaltikz}
\usepackage{calc}

\usepackage{hyperref}

\newcommand{\R}{\mathbb{R}} 

\newtheorem{remark}{Remark}

\begin{document}

\title{Error Estimates for Hyperbolic Scaling Limits of Linear Kinetic Models on Networks}

\author{Axel Klar\footnotemark[1] 
 \and Yizhou Zhou\footnotemark[2] }
\footnotetext[1]{RPTU  Kaiserslautern, Department of Mathematics, 67663 Kaiserslautern, Germany 
  (klar@rptu.de)}
\footnotetext[2]{Corresponding author, IGPM, RWTH Aachen University, D-52062 Aachen, Germany (zhou@igpm.rwth-aachen.de)}
 
\date{}


\maketitle

\begin{abstract}
This paper studies linear discrete kinetic models on networks and their asymptotic behavior in the small Knudsen number limit. For coupling conditions at an $n$-edge junction under a symmetric formulation, we introduce a change of variables that reformulates the system into $n$ independent initial–boundary value problems. The asymptotic expansions are then constructed and rigorously justified by an error estimate based on the energy method.
\end{abstract}

{\bf Keywords.} 
Kinetic equations on networks, coupling condition, kinetic interface layer, kinetic half-space problem.

{\bf AMS Classification.}  
82B40, 90B10, 35L50


\section{Introduction}

Mathematical models for flows on networks arise in a variety of practical applications, including traffic flow, gas transport in pipelines, supply chain dynamics, and blood circulation, see. e.g. \cite{BCGH} and the references therein. Such systems are typically described by partial differential equations defined on networks, where edges represent one-dimensional flow channels and nodes correspond to junctions. In the modeling, the formulation of coupling conditions at the junctions plays a key role for capturing the effective dynamics on complex networked structures.

In a macroscopic framework, coupling conditions have been extensively investigated for various classes of partial differential equations, including drift–diffusion, scalar hyperbolic, and hyperbolic systems such as the wave and Euler-type equations (see, e.g., \cite{BGKS,CC-17,BNR,BHK-06,BHK-06-Euler,CHS-08,ALM-10,EK-18,VZ-09,BCG-10,HKP-07,CGP-05,Gara-10,GN,GPS}). On the other hand, on a mesoscopic scale, coupling conditions for kinetic equations on networks have been also addressed in a number of studies (see \cite{FT-15,HM-09,BKKP}).  General and accurate procedures for linear kinetic equations were subsequently developed in \cite{BK-18-SISC,ABEK-24}, which are motivated by the classical derivation of kinetic slip boundary conditions for macroscopic equations (see \cite{BSS-84,BLP-79,Golse-92,Golse-08,UYY}). For the network problem of relaxation systems, we refer to \cite{MR4622941,BNR2}.
Particularly, coupling conditions for macroscopic equations on networks were derived in our previous works \cite{BDKZ1,BDKZ2} from underlying discrete kinetic models via an asymptotic analysis of the behavior near the network nodes. The main goal of the present study is to rigorously justify the validity of such derivations.

In the present work, we restrict our analysis to a discrete velocity
formulation of the kinetic equations. The resulting discrete kinetic
models can be written as first-order hyperbolic relaxation systems.
Accordingly, our approach builds on the theory of boundary conditions
for relaxation systems \cite{Yong-99-IUMJ,ZY-22,ZY-chara}, while being
adapted to the specific structure of network junctions. 

Specifically, we study two linear discrete kinetic models distinguished
by their collision operators. Despite their similar form, the two
models exhibit notably different asymptotic structures near a network
junction. The first model leads to a noncharacteristic limiting
problem, whereas the second may give rise to a characteristic boundary.
Consequently, depending on the collision operator and the boundary
condition, different types of boundary layers arise, including kinetic
and viscous layers. The corresponding formal expansions and effective
macroscopic coupling conditions were obtained in
\cite{BDKZ1,BDKZ2}. The main objective of the present work is to
justify these expansions rigorously.

We consider a symmetric coupling condition at a junction of $n$ edges.
A linear transformation of the edge variables decomposes the coupled
network problem into $n$ independent initial--boundary value problems.
For each problem, we extend the formal expansions in
\cite{BDKZ1,BDKZ2} to the order required for the error analysis. For
dissipative boundary conditions, the asymptotic equations are solved
explicitly using the structure of the model. In the strictly
dissipative case, their solvability follows from the generalized Kreiss
condition of \cite{Yong-99-IUMJ}. Since the asymptotic approximation
satisfies the original boundary condition, the error has a homogeneous
boundary condition, and dissipativity controls the boundary term.
Estimating the residual terms then yields explicit convergence rates.

The paper is organized as follows. Section \ref{section2} introduces the kinetic equations with two types of collision operators and the coupling conditions for the kinetic network problems. Section \ref{section3}  focuses on the discrete velocity formulation of the kinetic equations and the discrete version of the coupling conditions. In particular, we perform a change of variables and derive $n$-independent IBVPs. Section \ref{section5} presents the formal asymptotic expansions of these IBVPs. Finally, Section \ref{Section6} provides a justification of the asymptotic approximations through an error estimate. The main result is established in Theorem \ref{thm:error-estimates}.

\section{Kinetic model on a network}\label{section2}
We consider a junction connecting \(n\) identical half-lines, all parametrized by \(x>0\) and oriented away from the junction \(x=0\), see Figure \ref{fig:node}. On the \(i\)-th edge, the particle distribution \(f^{(i)}=f^{(i)}(x,v,t)\) satisfies the one-dimensional linear kinetic model
\begin{align}\label{kinetic-eq}
    f^{(i)}_t + v f^{(i)}_x = \frac{1}{\epsilon} \mathcal{Q}(f^{(i)}).
\end{align}
In this work, we consider {two linearized collision operators}:
$$
\mathcal{Q}_1(f) = - \left(f - (\rho+vq) M(v)\right)
$$
and 
$$
\mathcal{Q}_2(f) = - \left(f - \Big(\rho+vq+\frac{1}{2}(v^2-1)(S-\rho)\Big) M(v)\right)
$$
with the Maxwellian
$$
M(v) = \frac{1}{\sqrt{2\pi}}\exp\left(-\frac{v^2}{2}\right).
$$
Here the density, mean flux and total energy are given by
$$
\rho(x,t)=\int_{\mathbb{R}}f(x,v,t)dv,\quad 
q(x,t)=\int_{\mathbb{R}}vf(x,v,t)dv,\quad 
S(x,t)=\int_{\mathbb{R}}v^2f(x,v,t)dv.
$$
\begin{remark}
The operator  $\mathcal{Q}_2(f) $  is the linearization of the classical  kinetic BGK operator, see \cite{cer}.
 $\mathcal{Q}_1(f) $ is the linearization of a BGK-type relaxation model relaxing to a Maxwellian with constant temperature.
\end{remark}
Note that the 
associated macroscopic limit equations  for $\epsilon \rightarrow 0$ are, in case of the operator $\mathcal{Q}_1$,
the wave equation 
\begin{align}
\begin{aligned}
\partial_t \rho + \partial_x q &=0\\
\partial_t q +  \partial_x  \rho  &=0.\ 
\end{aligned}
\label{eq:wave}
\end{align}
In case of the operator $\mathcal{Q}_2$ we obtain the linearized Euler equations or
 acoustic system
\begin{equation}
\begin{aligned}
\partial_t\rho+\partial_xq&=0, \\
\partial_t q+\partial_x S&=0,\\
\partial_tS +3\partial_xq&=0.
\end{aligned}
\label{eq:acoustic}
\end{equation}
\begin{figure}[h]
    \centering
    \begin{tikzpicture}[scale=1.3]
    \node[circle,draw,fill=black,inner sep=2pt] (C) at (0,0) {};

    \draw [->] (C) -- ++(150:1) node[midway,above left] {} node[pos=1.2] {1};
    \draw[->]  (C) -- ++(180:1) node[pos=1.2] {2};
    \draw[->]  (C) -- ++(210:1) node[pos=1.2] {3};
    \draw [->] (C) -- ++(330:1) node[pos=1.2] {4};

    \draw[dashed] (C) -- ++(0:1);

    \draw [->] (C) -- ++(30:1) node[pos=1.2] {$n$};
\end{tikzpicture}
    \caption{Node connecting $n$ edges.}
    \label{fig:node}
\end{figure}

The kinetic distributions on different edges are coupled through a kinetic transmission condition at the junction.
Because all edges are oriented away from the junction, \(f^{(i)}(t,0,-v)\), \(v>0\), represents particles arriving at the junction along edge \(i\), whereas \(f^{(i)}(t,0,v)\) represents particles leaving the junction. 

For the network problem,  a general linear  coupling condition for the kinetic equations is given by 
\begin{align}
f^{(i)}(0,v,t) = \sum_{j=1}^n \beta_{ij}f^{(j)} (0,-v,t), v >0\ , i= 1, \ldots ,n
\end{align}
with $\beta_{ij} \in \R^+$ parameters determining the distribution of particles arriving at the junction onto the outgoing ones.
The total mass in the system is conserved, if 
\begin{align}\label{eq:conservative_coupling_f}
\sum_{i=1}^n \beta_{ij} = 1,
\end{align}
since in this case the balance of fluxes,  i.e. $\sum_{j=1}^n \int v f^{(j)}(0,v) dv =0$, holds.

In the following,  we will consider the case of a node with symmetric coupling conditions, 
that means  $\beta_{ij} = \frac{1}{n-1}, i \neq j$  and $\beta_{ii} =0$, that means the condition 
\begin{align}
	\label{kincoup}
f^{(i)}(0,v,t) = \frac{1}{n-1}\sum_{k=1,k\neq i}^n f^{(k)}(0,-v,t),\qquad v>0,\quad i=1,...,n.
\end{align}

\begin{remark}
For simplicity, we restrict our analysis to the symmetric case. Under this condition, from a microscopic point of view, particles arriving
along one edge are distributed equally among all the other edges, and
direct reflection into the same edge is excluded. 
As will be discussed below the symmetry of the interface conditions can be exploited to structure the problem improving the readability considerably. The computations could be extended  to the general case at the expense of 
much more cumbersome computations and notation.
 \end{remark}

For brevity, we write $f^{(i)}(v)=f^{(i)}(t,0,v)$ in the derivation of coupling condition. 
From the coupling condition, it is not difficult to see that \cite{BDKZ1}
\begin{align}\label{bc2}
\sum_{i=1}^n f^{(i)}(v) = \sum_{i=1}^n f^{(i)}(-v).
\end{align}
Moreover, we have the following $(n-1)$ relations
\begin{align}
(n-1)f^{(1)}(v)+f^{(1)}(-v) =&~ (n-1)f^{(2)}(v)+f^{(2)}(-v) \nonumber \\[1mm]
=&~\cdots = (n-1)f^{(n)}(v)+f^{(n)}(-v).\label{bc1}
\end{align} 

In the following we give a rigorous asymptotic investigation  of a velocity discretized version of the kinetic network problem
for small values of  $\epsilon$. 
On each edge the kinetic solution is expanded in a multiscale asymptotic expansion based on the decomposition 
\begin{align}
f ^{(i)}(x,v,t) =\bar{f}^{(i)}(x,v,t)+\widehat{f}^{(i)}\left(z,v,t\right)+\widetilde{f}^{(i)}\left(y,v,t\right),\;  z=\frac{x}{\sqrt{\epsilon}},\; y=\frac{x}{\epsilon}
\end{align}
taking 
kinetic (of order $\epsilon$) and viscous (of order $\sqrt{\epsilon}$) layers at the nodes into account. These expansions on the single edges are coupled via the 
interface conditions (\ref{kincoup}). A formal asymptotic expansion of the network problem to the order required by the  convergence proof  is given in Section \ref{section5}.
Finally, in Theorem \ref{thm:error-estimates} in Section \ref{Section6} we prove convergence 
of the asymptotic expansion   towards the coupled kinetic solution on the network as $\epsilon \rightarrow 0$, as well as explicit error estimates. 

\begin{remark}
	The  coupling conditions for the macroscopic equations  on the network are  given  in \cite{BDKZ1} and \cite{BDKZ2}  by the balance of fluxes
	$$\sum_{i=1}^n q^{(i)}(0,t) =0$$
	and
	by  invariants at the nodes given by $\rho + \delta_1 q$  and $S + \delta_2 q$ for the case of the wave equation and the acoustic system, respectively.
This leads to $n-1$ equations 
	$$
	\rho^{(1)}(0,t)+\delta q^{(1)}(0,t) = \rho^{(2)}(0,t)+\delta q^{(2)}(0,t) = \dots = \rho^{(n)}(0,t)+\delta q^{(n)}(0,t)
	$$
	in case of the wave equation and the corresponding equations for the acoustic system.
	We note that this gives a sufficient number of coupling conditions even for the acoustic system,
	since there is
	no condition required for the
	degenerate characteristic associated to the zero eigenvalue of the acoustic system.
	Moreover, note that the numbers  $\delta_1$  and $\delta_2$  play a similar role as the so-called extrapolation length for kinetic boundary layers \cite{BSS-84}.
Remarks \ref{remark5-m}  and \ref{rem2} below relate the present investigation to these coupling conditions and prove rigorously  that the kinetic problem 
converges towards the  macroscopic network problem based on equations (\ref{eq:wave}) amd  (\ref{eq:acoustic})  with the coupling conditions described in this remark.
\end{remark}

\section{Reformulation of the kinetic network problem}\label{section3}

As mentioned previously, in this work, we study  discrete-velocity approximations of the kinetic model on an \(n\)-edge network using Gauss–Hermite points as discrete velocities. As summarized in Figure \ref{fig:problem-reduction}, each discrete-velocity model can be equivalently written as a moment system through an invertible Hermite moment transformation. We then apply a linear change of edge variables to decouple the network problem into one IBVP with a conservative boundary condition and \(n-1\) identical IBVPs with a dissipative boundary condition. These decoupled IBVPs form the starting point for the asymptotic expansions and error estimates in the subsequent sections.

\begin{figure}[htbp]
\centering
\begin{small}
\begin{tikzpicture}[
    node distance=5mm and 2mm,
    box/.style={
        rectangle,
        rounded corners=2pt,
        draw=black,
        thick,
        align=center,
        text width=4.6cm,
        minimum height=8mm
    },
    result/.style={
        rectangle,
        rounded corners=2pt,
        draw=black,
        thick,
        align=center,
        text width=3.5cm,
        minimum height=8mm
    },
    arrow/.style={
        -{Latex[length=2mm]},
        thick
    }
]

\node[box] (network)
{Kinetic equations on an\\
\(n\)-edge network};

\node[box, below=of network] (dvm)
{Gauss--Hermite\\
discrete-velocity model};

\node[box, below=of dvm] (moment)
{Equivalent moment systems\\
on an \(n\)-edge network};

\node[box, below=of moment] (transform)
{\(n\) decoupled systems};

\node[result, below left=of transform] (symmetric)
{One IBVP 
with conservative boundary condition (BC)};

\node[result, below right=of transform] (difference)
{\(n-1\) identical IBVPs\\
with strictly dissipative BC};

\draw[arrow] (network) -- (dvm);
\draw[arrow] (dvm) -- (moment);
\draw[arrow] (moment) --
node[right, align=left, font=\small]
{Linear change of edge variables}
(transform);
\draw[arrow] (transform) -- (symmetric);
\draw[arrow] (transform) -- (difference);

\end{tikzpicture}
\caption{Reduction of the kinetic network problem to two distinct
classes of half-line initial--boundary value problems.} 
\label{fig:problem-reduction}
\end{small}
\end{figure}

\subsection{Discrete-velocity model and moment systems}\label{section3.1}
In this subsection, we consider the kinetic equation on one edge. Since the equations \eqref{kinetic-eq} on different edges are identical and decoupled, we omit the superscript \(i\) in this subsection for brevity.

We introduce the new variable $\widetilde{f}=f/M(v)$ and write the kinetic equation \eqref{kinetic-eq} as
\begin{align}\label{kinetic-eq1}
    \widetilde{f}_t + v \widetilde{f}_x  = -\frac{1}{\epsilon} (\widetilde{f} - \widetilde{\mathcal Q}_j).
\end{align}
Here we consider two types of the collision operators $\widetilde{\mathcal Q}_1$ and $\widetilde{\mathcal Q}_2$. For the first case, we have
$$
\widetilde{\mathcal Q}_1 =  \rho+vq
,\qquad 
\rho=\int_{\mathbb{R}}M(v)\widetilde{f}(v)dv,\quad 
q=\int_{\mathbb{R}}vM(v)\widetilde{f}(v)dv.
$$
For the second type, we have
$$
\widetilde{\mathcal Q}_2 = \rho+vq+\frac{1}{2}(v^2-1)(S-\rho),\qquad S=\int_{\mathbb{R}}v^2M(v)\widetilde{f}(v)dv.
$$
\begin{remark}
The two collision operators lead to different equilibrium structures.
The operator \(\widetilde{\mathcal Q}_1\) preserves the first two moments, while \(\widetilde{\mathcal Q}_2\)
additionally preserves the energy-related moment. The resulting
equilibrium system is non-characteristic in the first case, but characteristic in the second case. In particular, the zero
characteristic speed arising in the second case leads to a different layer structure \cite{BDKZ1,BDKZ2}.
\end{remark}

Consider the discrete velocity method and denote $\widetilde{f}_k=\widetilde{f}(v_k)$. Here we take $v_k$ for $1\leq k\leq 2N$ as Gauss--Hermite quadrature nodes:
$$
v_1=-z_1,~v_2=-z_2,~...,~v_{N}=-z_N,~v_{N+1}=z_1,~v_{N+2}=z_2,~...,~v_{2N}=z_N,
$$
where $z_i$ are positive solutions of $\phi_{2N}(z)=0$. In this work, we assume that the number of velocities is even and $2N\geq 4$.

The discrete velocity equations are
\begin{equation}\label{eq:DVM}
\partial_t \widetilde{f}_k + v_k \partial_x \widetilde{f}_k = -\frac{1}{\epsilon}(\widetilde{f}_k-\mathcal{E}_j(v_k)),\qquad 1\leq k\leq 2N,
\end{equation}
where \(\mathcal{E}_j(v_k)\) denotes the discrete equilibrium associated with the
\(j\)-th collision operator, \(j=1,2\), and is given by
\begin{eqnarray*}
    \mathcal{E}_j(v_k) = \left\{\begin{array}{ll}
       \phi_0g_0+\phi_1(v_k)g_1,  & \qquad  j=1\\[3mm]
        \phi_0g_0+\phi_1(v_k)g_1 + \phi_2(v_k)g_2,  & \qquad  j=2.
    \end{array}\right.
\end{eqnarray*}
Here the moments are defined by
\begin{equation}\label{def:moments}
g_m=\sum_{k=1}^{2N}\phi_m(v_k)w_k\widetilde{f}_k
\end{equation}
with $w_k$ being the Gauss--Hermite quadrature weights and $\phi_m(v)$ being the orthogonal Hermite polynomials. These polynomials satisfy
\begin{equation}\label{hermitian}
\int_{\mathbb{R}}\phi_i(v)\phi_j(v)M(v)dv = \delta_{ij}.    
\end{equation}
Clearly, we have $\phi_0(v)=1$ and $\phi_1(v)=v$. The other $\phi_k(v)$ can be determined by the recurrence relation \begin{equation}\label{recurrence relation}
    v\phi_k(v) = \sqrt{k} \phi_{k-1} + \sqrt{k+1} \phi_{k+1}.
\end{equation}
For example, we compute $\phi_2(v)=(v^2-1)/\sqrt{2}$. 

Recall that the definition of moments \eqref{def:moments} also holds for $3\leq m\leq 2N-1$.
We denote $$G=(g_0,g_1,...,g_{2N-1})^T$$ where each order of moment is given by \eqref{def:moments}. 

Now we derive the moment system from the discrete velocity model. 
Multiplying $w_k\phi_m(v_k)$ on the left side of \eqref{eq:DVM}, summing over $k$ and using the recurrence relation \eqref{recurrence relation} yield:
\begin{equation*}
\partial_t g_m + \sqrt{m}\partial_x g_{m-1} + \sqrt{m+1}\partial_x g_{m+1} = -\frac{1}{\epsilon}\left(g_m- \sum_{k=1}^{2N}w_k\phi_m(v_k)\mathcal{E}_j(v_k)\right).
\end{equation*} 
Notice that, for $m=2N-1$, we use the property $\phi_{2N}(v_k)=0$ for any $k=1,2,...,2N$ and get 
$$
\partial_t g_{2N-1} +  \sqrt{2N-1}\partial_x g_{2N-2} 
= -\frac{1}{\epsilon} g_{2N-1}.
$$ 
Then it follows that

\textbf{(Type 1):} For the first type collision $Q_1$, the moment system reads
\begin{equation}\label{eq:GQ1}
G_t+AG_x=\frac{1}{\epsilon}Q_1G
\end{equation}
with $$Q_1=\text{diag}(0,0,-1,\dots,-1)$$ and
\begin{align}\label{matrix A and Q}
A=\begin{pmatrix}
    0 & \alpha_1 & 0 & 0 & \cdots & 0 \\[1mm]
    \alpha_1 & 0 & \alpha_{2} & 0 & \cdots & 0 \\[1mm]
    0 & \alpha_{2} & 0 & \alpha_{3} &  &  \\
    \vdots & & \ddots & \ddots & \ddots &  \\
    0 & \cdots & 0 & \alpha_{2N-2} & 0 & \alpha_{2N-1} \\
    0 & \cdots & 0 & 0 & \alpha_{2N-1} & 0 \\
\end{pmatrix},\qquad \alpha_p=\sqrt{p}. 
\end{align}

\textbf{(Type 2):} For the second type collision $Q_2$, the moment system reads
\begin{equation}\label{eq:GQ2}
G_t+AG_x=\frac{1}{\epsilon}Q_2G
\end{equation}
with the same $A$ defined in \eqref{matrix A and Q} and $$Q_2=\text{diag}(0,0,0,-1,\dots,-1).$$  

Now we discuss the relation between the moment variable $G$ and the original discrete velocity variable. Denote
\begin{equation}\label{DV-var}
F=
\begin{pmatrix}
    F_-\\[2mm]
    F_+
\end{pmatrix},\quad 
F_-=
\begin{pmatrix}
    \widetilde{f}_1\\
    \vdots\\
    \widetilde{f}_N
\end{pmatrix},\quad
F_+=
\begin{pmatrix}
    \widetilde{f}_{N+1}\\
    \vdots\\
    \widetilde{f}_{2N}
\end{pmatrix}.
\end{equation}
Due to the definition in \eqref{def:moments}, we have the relation
\begin{equation}\label{MM-var}
G=\begin{pmatrix}
    V_- & V_+
\end{pmatrix}
\begin{pmatrix}
    W & \\[1mm]
    & W 
\end{pmatrix}
\begin{pmatrix}
    F_- \\[1mm]
    F_+
\end{pmatrix}:=V 
\begin{pmatrix}
    WF_- \\[1mm]
    WF_+
\end{pmatrix}
\end{equation}
with $W=\text{diag}(w_1,w_2,...,w_N)$ and
$$
V_-=
\begin{pmatrix}
    \phi_0(v_1)& \cdots & \phi_0(v_N)  
    \\[1mm]
    \phi_1(v_1)& \cdots & \phi_1(v_N)\\[1mm]
    \vdots & \ddots & \vdots \\[1mm]
    \phi_{2N-1}(v_1)& \cdots & \phi_{2N-1}(v_N)
\end{pmatrix}
~~
V_+=\begin{pmatrix}
    \phi_0(v_{N+1}) & \cdots & \phi_0(v_{2N}) \\[1mm]
    \phi_1(v_{N+1}) & \cdots & \phi_1(v_{2N}) \\[1mm]
    \vdots & \ddots & \vdots \\[1mm]
     \phi_{2N-1}(v_{N+1}) & \cdots & \phi_{2N-1}(v_{2N})
\end{pmatrix}.
$$
Furthermore, we state 
\begin{lemma}[\cite{BDKZ1}]\label{S}
Let $\{\phi_k(v)~|~k=1,2,...\}$ be the orthogonal Hermite polynomials on $ \mathbb{R}$ with the weight function $M(v)$. Let $-v_N,\cdots,-v_1,v_1,\cdots,v_N$ be roots of $\phi_{2N}(v)=0$. Then we have 
    $$
    V^{-1}=\begin{pmatrix}
        W & \\
          & W
    \end{pmatrix}V^T.
    $$
\end{lemma}
\noindent Based on this lemma, \eqref{MM-var} becomes 
\begin{equation}\label{trans-moment-DVM}
F=V^TG.
\end{equation}

\subsection{Coupling conditions}\label{section4}
In this subsection, we derive the discrete velocity formulation corresponding to the coupling conditions \eqref{bc2} and \eqref{bc1}. 

We use the superscript $i$ to express the variable at each edge $1\leq i\leq n$. By using the notations in \eqref{DV-var} and \eqref{MM-var}, the discrete velocity variables and moments at $i$-th edge are denoted by $F^{(i)}$ and $G^{(i)}$ respectively. 

Evaluating the coupling conditions \eqref{bc2} and \eqref{bc1} at discrete points $v_k~(1\leq k\leq 2N)$ yields 
\begin{align}
    &\sum_{i=1}^n \Big(I_N, -I_N\Big)F^{(i)}(0,t) = 0,\\[1mm]
    &\Big(I_N, (n-1)I_N\Big)F^{(1)}(0,t) = \dots = \Big(I_N, (n-1)I_N\Big)F^{(n)}(0,t).
\end{align}
Here $I_N$ is the $N$-dimensional identity matrix.
Using the relation \eqref{trans-moment-DVM}, we have
\begin{align}
    &\sum_{i=1}^n B_1G^{(i)}(0,t) = 0, \label{CCG1}\\[1mm]
    &B_2G^{(1)}(0,t) = B_2G^{(2)}(0,t) = \dots = B_2G^{(n)}(0,t)\label{CCG2}
\end{align}
with 
\begin{equation}\label{B1B2}
B_1=\begin{pmatrix}
        I_N & -I_N
    \end{pmatrix}V^T,\qquad B_2=\begin{pmatrix}
        I_N & (n-1)I_N
    \end{pmatrix}V^T.    
\end{equation}

\subsection{Change of edge variables}

In this subsection, we exploit the specific structure of the coupling conditions to reformulate the problem as a set of $n$-decoupled initial–boundary value problems (IBVPs) by using a linear change of edge variables.

Motivated by the coupling conditions \eqref{CCG1} and \eqref{CCG2}, we introduce new variables 
\begin{align}\label{defUk}
    U^{(1)} = \sum_{i=1}^n G^{(i)},\qquad U^{(k)} = G^{(k)}-G^{(1)},\quad k=2,3,...,n.
\end{align}
The coupling conditions for $U^{(1)}$ and $U^{(k)}$ read
\begin{align}
&B_1U^{(1)}(0,t) = 0, \label{BC1}\\[2mm]
&B_2U^{(k)}(0,t) = 0, \quad k=2,3,...,n.\label{BC2}
\end{align}

Due to the linearity of the equation, we know that 

\textbf{(Type 1):} For the first type collision, $U^{(k)}$ satisfies
\begin{equation}\label{eq:eqQ1}
U^{(k)}_t+AU^{(k)}_x=\frac{1}{\epsilon}Q_1U^{(k)},\qquad x>0,\quad k=1,2,...,n.
\end{equation}

\textbf{(Type 2):} For the second type collision, $U^{(k)}$ satisfies
\begin{equation}\label{eq:eqQ2}
U^{(k)}_t+AU^{(k)}_x=\frac{1}{\epsilon}Q_2U^{(k)},\qquad x>0,\quad k=1,2,...,n.
\end{equation}
\begin{remark}
In the original moment variables $\{G^{(i)}\}_{i=1}^n$, the equations
on the individual edges are coupled through the conditions at the
junction. The change of variables \eqref{defUk} completely decouples
the network problem into independent IBVPs for
$\{U^{(i)}\}_{i=1}^n$. Moreover, the problems for $U^{(i)}$,
$i\geq2$, are identical. This simplification relies on the symmetric
form of the coupling conditions and does not hold in general. For a
general coupling condition, one may instead collect all edge variables
into a single vector
\[
\mathcal G
=
\bigl((G^{(1)})^T,\ldots,(G^{(n)})^T\bigr)^T
\]
and treat the coupled network problem as one larger IBVP.
\end{remark}
In this way, we formulate four initial-boundary value problems (IBVPs), distinguished by their source terms or boundary conditions, see Table \ref{table1}. 
\begin{table}[h!]
\centering
\begin{tabular}{|c|c|c|}
\hline
 & collision term $Q_1$ & collision term $Q_2$ \\[0.6mm] \hline
boundary matrix $B_1$ & IBVP 1: \eqref{eq:eqQ1}~\text{with}~\eqref{BC1} & IBVP 3: \eqref{eq:eqQ2}~\text{with}~\eqref{BC1}\\[0.6mm] \hline
boundary matrix $B_2$ & IBVP 2: \eqref{eq:eqQ1}~\text{with}~\eqref{BC2} & IBVP 4: \eqref{eq:eqQ2}~\text{with}~\eqref{BC2} \\[0.6mm] \hline
\end{tabular}
\caption{Four IBVPs}
\label{table1}
\end{table}

Next we discuss the properties of the boundary matrices $B_1$ and $B_2$. 
\begin{lemma}\label{lemma2.2}
For any $n\geq 2$, the boundary matrix $B_2$ satisfies the following dissipative condition \cite{BenzoniSerre}:
$$
    y^TAy\leq 0,\qquad \text{for any}~~ y\in \text{ker} (B_2).
$$
Particularly, for $n\geq 3$ the matrix $B_2$ satisfies the strictly dissipative condition. Namely, the strict inequality $y^TAy<0$ holds.
\end{lemma}
\begin{proof}
We express the kernel of $B_2$ as
    $$
    y=V\begin{pmatrix}
        (n-1)I_N\\[1mm]
        -I_N
    \end{pmatrix}x,\quad x\in\mathbb{R}^N \setminus\{0\} .
    $$
    Then we compute 
    $$
    y^TAy = x^T\begin{pmatrix}
        (n-1)I_N & -I_N
    \end{pmatrix}V^TAV\begin{pmatrix}
        (n-1)I_N\\[1mm]
        -I_N
    \end{pmatrix}x
    $$
Using Lemma \ref{S}, we have
$$
V^TAV=\begin{pmatrix}
        W^{-1} & \\
          & W^{-1}
    \end{pmatrix}(V^{-1}AV)=\begin{pmatrix}
        W^{-1} & \\
          & W^{-1}
    \end{pmatrix}\begin{pmatrix}
        -\Lambda & \\
          & \Lambda
    \end{pmatrix},
$$
where $\Lambda=\text{diag}(z_1,z_2,...,z_N)$.
Thus we obtain
    $$
    y^TAy = -(n^2-2n)x^TW^{-1}\Lambda x .
    $$ 
Recall that $W$ and $\Lambda$ are diagonal matrices with positive entries. Consequently, we find that $y^TAy\leq 0$ for any $n\geq 2$ and $y^TAy<0$ for any $n\geq 3$.
\end{proof}
On the other hand, for the matrix $B_1$, we have
\begin{lemma}\label{lemma2.3}
    For any $n$, the boundary matrix $B_1$ satisfies the following relation:
$$
    y^TAy=0,\qquad \text{for any}~~ y\in \text{ker} (B_1).
$$
\end{lemma}
\begin{proof}
We express the kernel of $B_1$ as
    $$
    y=V\begin{pmatrix}
        I_N\\[1mm]
        I_N
    \end{pmatrix}x,\quad x\in\mathbb{R}^N \setminus\{0\} .
    $$
    Then we compute 
    $$
    y^TAy = x^T\begin{pmatrix}
         I_N & I_N
    \end{pmatrix}V^TAV\begin{pmatrix}
        I_N\\[1mm]
        I_N
    \end{pmatrix}x
    $$
Similar to the previous lemma, we have
$$
V^TAV=\begin{pmatrix}
        W^{-1} & \\
          & W^{-1}
    \end{pmatrix}(V^{-1}AV)=\begin{pmatrix}
        W^{-1} & \\
          & W^{-1}
    \end{pmatrix}\begin{pmatrix}
        -\Lambda & \\
          & \Lambda
    \end{pmatrix},
$$
and therefore $y^TAy = 0$.
\end{proof}

\section{Asymptotic expansion}\label{section5}

In this section, we construct the asymptotic solutions corresponding to the four cases listed in Table \ref{table1}. It is worth noting that the leading-order asymptotic expansions have already been derived in \cite{BDKZ1,BDKZ2}. For the purpose of establishing error estimates in the next section, we extend these results by including higher-order terms so that the residuals become sufficiently small.

\subsection{First-type collision term $Q_1$: IBVP 1}\label{subsection21}
In this subsection, we denote $U=U^{(1)}$ and consider the IBVP
\begin{align}\label{IBVP1}
\left\{\begin{array}{l}
U_t+AU_x=Q_1U/\epsilon,\\[2mm]
B_1U(0,t)=0.
\end{array}\right.
\end{align}
Recall that $B_1=\begin{pmatrix}
        I_N & (n-1)I_N
    \end{pmatrix}V^T$.
Moreover, we use the following block-matrix notation in this subsection: 
\begin{align}
A&=\begin{pmatrix}
        A_{11} & A_{12}\\[1mm]
        A_{12}^T & A_{22}
    \end{pmatrix}\quad \text{with}\quad A_{11} = \begin{pmatrix}
        0 & \alpha_1 \\ 
        \alpha_1 & 0 
    \end{pmatrix}\quad
A_{12} = \begin{pmatrix}
        0 & 0 & \cdots & 0\\ 
        \alpha_2 & 0 & \cdots& 0
    \end{pmatrix} \label{IBVP1-A11}\\[2mm]
A_{22} &= \begin{pmatrix}
   0 & \alpha_{3} &  &  \\
   \alpha_3~ & 0 & \ddots &  \\
    & \ddots & \ddots & \alpha_{2N-1} \\[1mm]
    &  & \alpha_{2N-1} & 0 
\end{pmatrix}.\label{IBVP1-A22}
\end{align}

Consider the asymptotic solution
$$
U_\epsilon = \bar{U}_0(x,t) + \epsilon \bar{U}_1(x,t) = \begin{pmatrix}
    \bar{u}_0\\[1mm]
    \bar{v}_0
\end{pmatrix}(x,t)+\epsilon\begin{pmatrix}
    \bar{u}_1\\[1mm]
    \bar{v}_1
\end{pmatrix}(x,t).
$$
Here $\bar{u}_0$ and $\bar{u}_1$ represent the first two components of $\bar{U}_0$ and $\bar{U}_1$ while $\bar{v}_0$ and $\bar{v}_1$ are the last $(2N-2)$ components of $\bar{U}_0$ and $\bar{U}_1$.
In other words, for the first boundary value problem with collision term $Q_1$ we assume that there is no boundary layer.
Substituting the asymptotic expansion into the equation \eqref{IBVP1}, we match each order of $\epsilon$ and derive the equations 
\begin{align}
O(\epsilon^{-1}):\qquad& \bar{v}_{0}=0 \label{IBVP1-bar-v0}\\[2mm]
O(\epsilon^{0}):\qquad& \partial_t\bar{u}_{0}
+ A_{11} \partial_x \bar{u}_{0} = 0\label{IBVP1-bar-u0}\\[2mm]
O(\epsilon^{0}):\qquad& \bar{v}_{1}=-A_{12}^T\partial_x\bar{u}_0 \label{IBVP1-bar-v1}\\[2mm]
O(\epsilon^1):\qquad& \partial_t\bar{u}_{1}
+ A_{11} \partial_x \bar{u}_{1} = - A_{12} \partial_x \bar{v}_{1}.\label{IBVP1-bar-u1}
\end{align}
\begin{remark}	
At leading order, equation \eqref{IBVP1-bar-u0} corresponds to the wave equation studied in \cite{BDKZ1}:
recalling from Section \ref{section4}, that $U$ satisfies the same equation as the moment variable $G=(g_0,g_1,\dots,g_{2N-1})$
and using $\rho = g_0$ and $q = g_1$, the equation reads
\begin{align*}
\left\{
\begin{array}{l}
\rho_t + q_x = 0,\\[2mm]
q_t + \rho_x = 0.
\end{array}\right.
\end{align*}
\end{remark}

The boundary condition in \eqref{IBVP1} implies 
$$
\begin{pmatrix}
    I_N & -I_N
\end{pmatrix}V^T
\begin{pmatrix}
    \bar{u}_0 \\[1mm]
    \bar{v}_0 
\end{pmatrix}(0,t)=0,\qquad
\begin{pmatrix}
    I_N & -I_N
\end{pmatrix}V^T
\begin{pmatrix}
    \bar{u}_1 \\[1mm]
    \bar{v}_1 
\end{pmatrix}(0,t)=0.
$$
Thanks to the property of Hermite polynomials, we have
\begin{align*}
\begin{pmatrix}
    I_N & -I_N
\end{pmatrix}V^T = -2(0, \Phi_1, 0, \Phi_3,\dots,0,\Phi_{2N-1})   
\end{align*}
where $\Phi_k=(\phi_k(v_{N+1}),\phi_k(v_{N+2}),\dots,\phi_k(v_{2N}))^T$ for $0\leq k\leq 2N-1$. Then it follows that 
\begin{align*}
(0, \Phi_1)\bar{u}_0 = 0,\qquad 
(0, \Phi_1)\bar{u}_1 + (0, \Phi_3,\dots,0,\Phi_{2N-1})\bar{v}_1= 0.
\end{align*}
By \eqref{IBVP1-bar-v1}, we have $\bar{v}_1=-A_{12}^T\partial_x\bar{u}_0$. Notice that $(0, \Phi_3,\dots,0,\Phi_{2N-1})A_{12}^T=0$.
Then the above relations become
\begin{align}\label{reducedBC1}
(0, \Phi_1)\bar{u}_0 = 0,\qquad 
(0, \Phi_1)\bar{u}_1 = 0.
\end{align}
Consider the characteristic variables $\beta_{k+},\beta_{k-}$ such that 
$$
\bar{u}_k = P_+\beta_{k+}+P_-\beta_{k-},\qquad k=0,1
$$ 
where $P_+$ and $P_-$ are eigenvectors of $A_{11}$ associated with positive and negative eigenvalue respectively. By computation, we know that 
$$
P_+ = \frac{1}{\sqrt{2}}\begin{pmatrix}
    1\\
    1
\end{pmatrix},\qquad 
P_- = \frac{1}{\sqrt{2}}\begin{pmatrix}
    1\\
    -1
\end{pmatrix}.
$$
Therefore, the boundary condition in \eqref{reducedBC1} can be expressed as
$$
\Phi_1\beta_{k+} = \Phi_1\beta_{k-}
\quad \Leftrightarrow \quad \beta_{k+} = \beta_{k-}.
$$
These are the reduced boundary conditions for hyperbolic systems \eqref{IBVP1-bar-u0} and \eqref{IBVP1-bar-u1}. Note that no boundary layer correction term is needed in this case.

\begin{remark}
    Note that, for the leading order, the boundary condition $\beta_{0+}=\beta_{0-}$ means that the second component is equal to $0$. Recall the setting $U=U^{(1)}$ in this subsection and the definition of $U^{(1)}$ in \eqref{defUk}. Thus, this boundary condition implies the balance of fluxes \cite{BDKZ1}
    $$
    \sum_{i=1}^nq^{(i)}(0,t) = 0
    $$
    for the limit problem.
\end{remark}

\subsection{First-type collision term $Q_1$: IBVP 2}

In this subsection, we denote $U=U^{(2)}$ and consider the IBVP
\begin{align}\label{IBVP2}
\left\{\begin{array}{l}
U_t+AU_x=Q_1U/\epsilon,\\[2mm]
B_2U(0,t)=0.
\end{array}\right.
\end{align}
Moreover, we use the same block-matrix notation as that in \eqref{IBVP1-A11} and \eqref{IBVP1-A22}. 
Consider the asymptotic solution
$$
U_\epsilon = \bar{U}_0(x,t)+\epsilon \bar{U}_1(x,t)+\widetilde{U}_0\left(y,t\right)+\epsilon \widetilde{U}_1\left(y,t\right),\qquad y=\frac{x}{\epsilon}.
$$
Here $\bar{U}_{0}$ and $\bar{U}_{1}$ are the outer solutions while $\widetilde{U}_0$ and $\widetilde{U}_1$ represent the kinetic boundary-layer correction terms. Thus, for the boundary value problem with collision term $Q_1$ we assume a kinetic boundary layer, but no viscous boundary layer. By the matching principle, we have
$$
\widetilde{U}_0(\infty,t)=\widetilde{U}_1(\infty,t)=0.
$$
The outer solutions 
$$
\bar{U}_0 = \begin{pmatrix}
    \bar{u}_0\\
    \bar{v}_0
\end{pmatrix}\qquad 
\bar{U}_1 = \begin{pmatrix}
    \bar{u}_1\\
    \bar{v}_1
\end{pmatrix}
$$ 
satisfy the same equations as those in \eqref{IBVP1-bar-v0}---\eqref{IBVP1-bar-u1}. We substitute the boundary-layer correction term $\widetilde{U}_0+\epsilon\widetilde{U}_1$ into the equation \eqref{IBVP2} and match each order of $\epsilon$ to obtain
\begin{align}
    O(\epsilon^{-1});\qquad &A\partial_y\widetilde{U}_0 = Q_1\widetilde{U}_0\label{IBVP2-U0tilde} \\[1mm] 
    O(\epsilon^{0}):\qquad &\partial_t\widetilde{U}_0 + A\partial_y\widetilde{U}_1 = Q_1\widetilde{U}_1  .\label{IBVP2-U1tilde}
\end{align}
For $k=0,1$, we denote $\widetilde{u}_k$ by the first two components of $\widetilde{U}_k$ and $\widetilde{v}_k$ the last $(2N-2)$ components of $\widetilde{U}_k$. 
By \eqref{IBVP2-U0tilde}, we have
\begin{align*}
    \begin{pmatrix}
        A_{11} & A_{12}\\[1mm]
        A_{12}^T & A_{22}
    \end{pmatrix}\partial_y
    \begin{pmatrix}
        \widetilde{u}_{0}\\[1mm]
        \widetilde{v}_{0}
    \end{pmatrix} = \begin{pmatrix}
        0\\[1mm]
        -\widetilde{v}_{0}
    \end{pmatrix},\quad 
    \begin{pmatrix}
        A_{11} & A_{12}\\[1mm]
        A_{12}^T & A_{22}
    \end{pmatrix}\partial_y
    \begin{pmatrix}
        \widetilde{u}_{1}\\[1mm]
        \widetilde{v}_{1}
    \end{pmatrix} = \begin{pmatrix}
        0\\[1mm]
        -\widetilde{v}_{1}
    \end{pmatrix} - \begin{pmatrix}
        \partial_t\widetilde{u}_{0}\\[1mm]
        \partial_t\widetilde{v}_{0}
    \end{pmatrix}.
\end{align*}
Since $A_{11}$ is invertible, we have 
\begin{align}
&\widetilde{u}_{0} = -A_{11}^{-1}A_{12}\widetilde{v}_{0}\label{IBVP2-aeq1}\\[2mm]
&(A_{22}-A_{12}^TA_{11}^{-1}A_{12})\partial_y\widetilde{v}_0 = -\widetilde{v}_0\label{IBVP2-ODE1}\\[1mm]
&\widetilde{u}_{1} = -A_{11}^{-1}A_{12}\widetilde{v}_{1} + \int_y^{\infty}A_{11}^{-1}\partial_t \widetilde{u}_0dy \label{IBVP2-aeq2}\\[1mm]
&(A_{22}-A_{12}^TA_{11}^{-1}A_{12})\partial_y\widetilde{v}_1 = -\widetilde{v}_1 + \widetilde{h}\label{IBVP2-ODE2}
\end{align}
with $\widetilde{h} = A_{12}^TA_{11}^{-1}\partial_t\widetilde{u}_0-\partial_t\widetilde{v}_0$.
By computation, we know that $A_{12}^TA_{11}^{-1}A_{12}=0$. 
Due to the classical theory of ordinary differential equations \cite{ODE}, the initial data $\widetilde{v}_{k}(0)$ for the bounded solutions of \eqref{IBVP2-ODE1} and \eqref{IBVP2-ODE2} should be given on the stable subspace:
$$
\widetilde{v}_{k}(0)= \widetilde{R}_{-} \gamma_{k-},\qquad k=0,1.
$$
Here $\widetilde{R}_{-}$ represents the eigenvectors of $-A_{22}^{-1}$ associated with negative eigenvalues. 

Substituting the asymptotic solution into the boundary condition \eqref{IBVP2} yields
\begin{align}\label{IBVP2-BC}
B_2\begin{pmatrix}
    \bar{u}_0+\widetilde{u}_0\\[1mm]
    \bar{v}_0+\widetilde{v}_0
\end{pmatrix}(0,t)=0,\qquad 
B_2\begin{pmatrix}
    \bar{u}_1+\widetilde{u}_1\\[1mm]
    \bar{v}_1+\widetilde{v}_1
\end{pmatrix}(0,t)=0.
\end{align}
For further derivation, we consider the characteristic variables
$\bar{u}_k = P_+\beta_{k+}+P_-\beta_{k-}$ with $k=0,1$. 
Using the outer equations and the equation \eqref{IBVP2-aeq1}, the first equation can be written as
\begin{align}\label{eq:BClinear}
B_2\begin{pmatrix}
        P_{+} & -A_{11}^{-1}A_{12}\widetilde{R}_{-}\\[1mm]
        0 & \widetilde{R}_{-}
    \end{pmatrix}
    \begin{pmatrix}
    \beta_{0+}\\[1mm]
    \gamma_{0-}
\end{pmatrix}=-B_2\begin{pmatrix}
    P_{-}\beta_{0-}\\[1mm]
    0
\end{pmatrix}.
\end{align}
For this algebraic equation, we have 
\begin{proposition}\label{prop1}
    Assume that the boundary matrix $B_2$ satisfies the strictly dissipative condition, then the coefficient matrix $$
    B_2\begin{pmatrix}
        P_{+} & -A_{11}^{-1}A_{12}\widetilde{R}_{-}\\[1mm]
        0 & \widetilde{R}_{-}
    \end{pmatrix}
    $$ in the equation \eqref{eq:BClinear} is invertible.
\end{proposition}

The proof of this proposition can be founded in \cite{Yong-99-IUMJ} (Lemma 3.4) which is based on the theory of generalized Kreiss condition proposed therein. Thanks to this result, we have well-posed reduced boundary condition for the outer solution $\bar{u}_{0}$ and the initial condition for the equation \eqref{IBVP2-ODE1}.
Similarly, we can obtain the reduced boundary condition for $\bar{u}_{1}$ and the initial condition for the equation \eqref{IBVP2-ODE2}. Notice that in the derivation, $\widetilde{U}_0$ is regarded as given functions. The exceptional case $n=2$, for which $B_2$ is dissipative but not strictly dissipative, has been analyzed separately in our previous work \cite{BDKZ1}, where the required solvability was established.

\begin{remark}\label{remark5-m}
    For the purpose of the present work, it suffices to prove the solvability of the boundary condition \eqref{eq:BClinear} and complete the construction of the asymptotic expansions. In \cite{BDKZ1}, the explicit computation gives the coupling conditions for the limit equations as: 
    $$
    \rho^{(1)}(0,t)+\delta q^{(1)}(0,t) = \rho^{(2)}(0,t)+\delta q^{(2)}(0,t) = \dots = \rho^{(n)}(0,t)+\delta q^{(n)}(0,t).
    $$
    Here we explain how the
solvability result in Proposition~\ref{prop1} justifies the derivation in \cite{BDKZ1}. In \cite{BDKZ1}, the coupling conditions are derived from the fact that the matrix 
$$
\begin{pmatrix}I_N&(n-1)I_N\end{pmatrix}V^TT = \begin{pmatrix}
I_N&(n-1)I_N
\end{pmatrix}V^T
\begin{pmatrix}
T_{11}&T_{12}\\[1mm]
0 & \widetilde{R}_-
\end{pmatrix}
$$ has full-rank (see Section 6.2 and 7.3 of \cite{BDKZ1}). Here we have 
$$
T_{11} = \frac{1}{\sqrt{2}}\begin{pmatrix}
    1 & 0\\
    0 & 1
\end{pmatrix},\qquad T_{12} = -\sqrt{2}\begin{pmatrix}
    e_1^T\widetilde{R}_-\\
    0 
\end{pmatrix},
$$
where $e_1^T=(1,0,\dots,0)$. Notice that the notation $\widetilde{R}_-$ corresponds to $R_2^+$ in \cite{BDKZ1}; the difference in notation results from describing the
same stable subspace in terms of the negative eigenvalues of
$-A_{22}^{-1}$ or, equivalently, the positive eigenvalues of $A_{22}$. Proposition~\ref{prop1} implies that the matrix 
$$
\Psi = B_2\begin{pmatrix}
        P_{+} & P_- & -A_{11}^{-1}A_{12}\widetilde{R}_{-}\\[1mm]
        0 & 0 & \widetilde{R}_{-}
    \end{pmatrix} 
    $$
    has full rank. Therefore, it follows from the relation  $T_{12}=-A_{11}^{-1}A_{12}\widetilde{R}_-$ and 
    $$\Psi = \begin{pmatrix}
I_N,~ (n-1)I_N
\end{pmatrix}V^T
\begin{pmatrix}
T_{11}&T_{12}\\[1mm]
0 & \widetilde{R}_-
\end{pmatrix}
\begin{pmatrix}
        T_{11}^{-1}(P_{+},~ P_-) & 0 \\[1mm]
        0 & I
        \end{pmatrix}
    $$ 
    that $\begin{pmatrix}
I_N,~ (n-1)I_N
\end{pmatrix}V^TT$ also has full-rank. This verifies the rank condition underlying the
derivation of the macroscopic invariant in \cite{BDKZ1}. We refer to
that work for the explicit computation of the coefficient $\delta$.
\end{remark}

\subsection{Second-type collision term $Q_2$: IBVP 3}\label{subsection4.3}
In this subsection, we denote $U=U^{(1)}$ and consider the IBVP
\begin{align}\label{IBVP3}
\left\{\begin{array}{l}
U_t+AU_x=Q_2U/\epsilon,\\[2mm]
B_1U(0,t)=0.
\end{array}\right.
\end{align}
Moreover, in this subsection we denote 
\begin{align}
    A&=\begin{pmatrix}
        A_{11} & A_{12}\\[1mm]
        A_{12}^T & A_{22}
    \end{pmatrix}\quad \text{with}\quad A_{11} = \begin{pmatrix}
        0 & \alpha_1 & 0\\ 
        \alpha_1 & 0 & \alpha_2\\ 
        0 & \alpha_2 & 0
    \end{pmatrix}\label{IBVP3-A11}\\[2mm]
A_{12} &= \begin{pmatrix}
        0 & 0 & \cdots\\ 
        0 & 0 & \cdots\\ 
        \alpha_3 & 0 & \cdots
    \end{pmatrix}\quad
A_{22} = \begin{pmatrix}
   0 & \alpha_{4} &  &  \\
   \alpha_4~ & 0 & \ddots &  \\
    & \ddots & \ddots & \alpha_{2N-1} \\[1mm]
    &  & \alpha_{2N-1} & 0 
\end{pmatrix}.\label{IBVP3-A22}
\end{align}
Consider the asymptotic solution with the viscous boundary-layer correction term
\begin{align}
U_\epsilon(x,t) = \bar{U}_0(x,t) + \epsilon \bar{U}_1(x,t) +  \sqrt{\epsilon} \widehat{U}_0(z,t) + \epsilon \widehat{U}_1(z,t),\qquad z=x/\sqrt{\epsilon}.
\end{align}
In contrast to the first boundary value problem with collision term $Q_1$ we expect  here a viscous layer, but still no kinetic layer. 
The matching principle requires
$$
\widehat{U}_0(\infty,t)=\widehat{U}_1(\infty,t)=0.
$$
Denote $\bar{u}_k$ by the first three components of $\bar{U}_k$ and $\bar{v}_k$ the last $(2N-3)$ components of $\bar{U}_k$. For the outer equation, we match the coefficients of $\epsilon$ in the equation 
$$
\partial_t\begin{pmatrix}
    \bar{u}_0+\epsilon\bar{u}_1\\[1mm]
    \bar{v}_0+\epsilon\bar{v}_1
\end{pmatrix}
+ \begin{pmatrix}
    A_{11} & A_{12}\\[1mm]
    A_{12}^T & A_{22}
\end{pmatrix}
\partial_x\begin{pmatrix}
    \bar{u}_0+\epsilon\bar{u}_1\\[1mm]
    \bar{v}_0+\epsilon\bar{v}_1
\end{pmatrix}=
\frac{1}{\epsilon}\begin{pmatrix}
    0\\[1mm]
    -\bar{v}_0-\epsilon\bar{v}_1
\end{pmatrix}
$$
to obtain:
\begin{align}
&O(\epsilon^{-1}):\qquad \bar{v}_{0} = 0\label{IBVP3-v0}\\[2mm]
&O(\epsilon^{0}):~\quad\quad \partial_t\bar{u}_0 + A_{11} \partial_x \bar{u}_0 = 0\label{IBVP3-u0}\\[2mm]
&O(\epsilon^{0}):~\quad\quad \bar{v}_{1} = -A_{12}^T \partial_x \bar{u}_0 \label{IBVP3-v1}\\[2mm]
&O(\epsilon^{1}):~\quad\quad \partial_t\bar{u}_1 + A_{11} \partial_x \bar{u}_1 = -A_{12}\partial_x \bar{v}_1.\label{IBVP3-u1}
\end{align}
For hyperbolic systems \eqref{IBVP3-u0} and \eqref{IBVP3-u1}, we introduce the characteristic variable $\beta_{k+}$, $\beta_{k0}$ and $\beta_{k-}$ with $k=0,1$ such that
$$
\bar{u}_k=P_{+}\beta_{k+} + P_{0}\beta_{k0} + P_{-}\beta_{k-}.
$$
Here $P_+$, $P_-$ and $P_0$ are eigenvectors of $A_{11}$ associated with positive, negative and zero eigenvalue. By computation, we have
\begin{equation}\label{IBVP3-R1+R10R1-}
P_{+}^T=\frac{1}{\sqrt{2}\lambda}
\begin{pmatrix}
\alpha_1,~\lambda,~ \alpha_2 
\end{pmatrix},
~~
P_{0}^T=\frac{1}{\lambda}\left(\alpha_2,0,-\alpha_1\right),
~~
P_{-}^T=\frac{1}{\sqrt{2}\lambda}
\begin{pmatrix}
\alpha_1,-\lambda,~ \alpha_2 
\end{pmatrix}
\end{equation}
with $\lambda=\sqrt{\alpha_1^2+\alpha_2^2}$ the positive eigenvalue of $A_{11}$. By the classical theory for hyperbolic equations, at the boundary point $x=0$, one boundary condition should be given for the variable $\beta_{k+}$.

\begin{remark}
	The outer equation \eqref{IBVP3-u0} at the leading order is the limiting equation (2.4) in \cite{BDKZ2}:
recall Section \ref{section4} that $U$ satisfies the same equation as the moment variable $G=(g_0,g_1,\dots,g_{2N-1})$. 
 Then we denote $\rho= g_0$, $q= g_1$ and $S=\sqrt{2}g_2+\rho$ to obtain the acoustic system
\begin{align*}
\left\{
\begin{array}{l}
\partial_t \rho+\partial_x q=0,\\[2mm]
\partial_t q+\partial_x S=0,\\[2mm]
\partial_t S+3\partial_x q=0.
\end{array}\right.
\end{align*}
\end{remark}


Denote $\widehat{u}_k$ by the first three components of $\widehat{U}_k$ and $\widehat{v}_k$ the last $(2N-3)$ components of $\widehat{U}_k$. Matching each order of $\epsilon$ in the equation
$$
\partial_t\begin{pmatrix}
    \sqrt{\epsilon}\widehat{u}_0+\epsilon\widehat{u}_1\\[1mm]
    \sqrt{\epsilon}\widehat{v}_0+\epsilon\widehat{v}_1
\end{pmatrix}
+ \frac{1}{\sqrt{\epsilon}}\begin{pmatrix}
    A_{11} & A_{12}\\[1mm]
    A_{12}^T & A_{22}
\end{pmatrix}
\partial_z\begin{pmatrix}
    \sqrt{\epsilon}\widehat{u}_0+\epsilon\widehat{u}_1\\[1mm]
    \sqrt{\epsilon}\widehat{v}_0+\epsilon\widehat{v}_1
\end{pmatrix}=
\frac{1}{\epsilon}\begin{pmatrix}
    0\\[1mm]
    -\sqrt{\epsilon}\widehat{v}_0-\epsilon\widehat{v}_1
\end{pmatrix},
$$
we obtain
\begin{align}
&O(\epsilon^{-1/2}):\qquad \widehat{v}_{0}=0,\label{IBVP3-vis-1}\\[1mm]
&O(\epsilon^0):\qquad\quad~
A_{11}\partial_z\widehat{u}_{0}
 = 0,  \label{IBVP3-vis-2}\\[1mm]
&O(\epsilon^0):\qquad \quad~
\widehat{v}_{1} = -A_{12}^T\partial_z\widehat{u}_{0}, \label{IBVP3-vis-3}\\[1mm]
&O(\epsilon^{1/2}):\qquad~ \partial_t 
        \widehat{u}_{0} + A_{11}
    \partial_z\widehat{u}_{1}+A_{12}
    \partial_z\widehat{v}_1 = 0. \label{IBVP3-vis-4}
\end{align}
Now we write $P_1=(P_+,P_-)$ to represent the eigenvectors associated with nonzero eigenvalues. Note that $(P_1,P_0)$ is orthogonal and therefore $I=P_1P_1^T+P_0P_0^T$. We multiply $P_1^T$ on the left of \eqref{IBVP3-vis-2} to obtain
$\Lambda_1 P_1^T\partial_z\widehat{u}_0 = 0$.
Here $\Lambda_1=\text{diag}(\lambda,-\lambda)$ represents the nonzero eigenvalues of $A_{11}$. Since $\widehat{u}_0(\infty,t)=0$, we know that 
\begin{align}\label{IBVP-vis-5}
P_1^T\widehat{u}_0=0.    
\end{align}
Then we multiply $P_0^T$ on the left of \eqref{IBVP3-vis-4} and use \eqref{IBVP3-vis-3} to get 
\begin{align}\label{IBVP3-vis-parabolic}
\partial_t(P_0^T\widehat{u}_{0}) - (P_0^TA_{12}A_{12}^TP_0)
\partial_{zz}(P_0^T\widehat{u}_0) = 0.
\end{align}
\begin{remark}
Recall that $P_0^TA_{12}A_{12}^TP_0=(\alpha_1^2\alpha_3^2)/\lambda^2=1$. The above equation is the heat equation (3.8) derived in \cite{BDKZ2}. Indeed, we denote $\widehat{u}_0=(\widehat{u}_{00},\widehat{u}_{01},\widehat{u}_{02})^T$ and compute 
$$
P_0^T\widehat{u}_0=\frac{1}{\lambda}(\alpha_2\widehat{u}_{00}-\alpha_1\widehat{u}_{02})=\frac{1}{\sqrt{3}}(\sqrt{2}\widehat{u}_{00}-\widehat{u}_{02}).
$$ 
Then with the setting $\widehat{\rho}=\sqrt{2}\widehat{u}_{00}$, $\widehat{q}=\sqrt{2}\widehat{u}_{01}$ and $\widehat{S}=2\widehat{u}_{02}+\widehat{\rho}$ we have $$
\partial_t(P_0^T\widehat{u}_{0})-\partial_{xx}(P_0^T\widehat{u}_{0})=0\quad \Rightarrow \quad \partial_t(\widehat{S}-3\widehat{\rho})-\partial_{xx}(\widehat{S}-3\widehat{\rho})=0,
$$
which is exactly the equation (3.8) in \cite{BDKZ2}.
\end{remark}

Furthermore, we multiply $P_1^T$ on the left of \eqref{IBVP3-vis-4} to get 
\begin{align}\label{IBVP3-vis-6}
P_1^T\widehat{u}_1=-\Lambda_1^{-1}P_1^TA_{12}\widehat{v}_1.
\end{align}
Here we use the fact $P_1^T\widehat{u}_0=0$ and the matching principle $\widehat{U}_1(\infty,t)=0$. Notice that the term $P_0^T\widehat{u}_1$ is not determined by the equations \eqref{IBVP3-vis-1}---\eqref{IBVP3-vis-4}. For our purpose, $P_0^T\widehat{u}_1$ can be given arbitrarily and the validity will be shown in Section \ref{Section6}.

Now we turn to consider the boundary conditions. Substituting the asymptotic solution into the boundary condition yields
\begin{align}\label{IBVP3-BC}
B_1\begin{pmatrix}
\bar{u}_0+\sqrt{\epsilon}\widehat{u}_0+\epsilon\bar{u}_1+\epsilon\widehat{u}_1\\[1mm]
\bar{v}_0+\sqrt{\epsilon}\widehat{v}_0+\epsilon\bar{v}_1+\epsilon\widehat{v}_1
\end{pmatrix}(0,t)=0. 
\end{align}
Denote $\Phi_k=(\phi_k(v_{N+1}),\dots,\phi_k(v_{2N}))^T$ for $0\leq k\leq 2N-1$. Then the boundary matrix $B_1$ can be written as
$B_1=-2(0,\Phi_1,0,\Phi_3,\dots,0,\Phi_{2N-1})$. Using \eqref{IBVP3-v0}, we write the $O(\epsilon^0)$ term in \eqref{IBVP3-BC} as
$$
-2(0,\Phi_1,0,\Phi_3,\dots,0,\Phi_{2N-1})\begin{pmatrix}
    \bar{u}_0\\ 
    0
\end{pmatrix}=-2(0,\Phi_1,0)\bar{u}_0=0. 
$$
Using the characteristic decomposition $\bar{u}_0=P_{+}\beta_{0+} + P_{0}\beta_{00} + P_{-}\beta_{0-}$ with $P_+,P_-,P_0$ given in \eqref{IBVP3-R1+R10R1-}, we have 
$$
\Phi_1 \beta_{0+} = \Phi_1 \beta_{0-}\quad \Leftrightarrow \quad \beta_{0+} = \beta_{0-}.
$$
This yields the reduced boundary condition for the hyperbolic system \eqref{IBVP3-u0}.
Note that the boundary condition $\beta_{0+}=\beta_{0-}$ means that the second component of $\bar{u}_0$ is equal to $0$, which implies the balance of fluxes for the limit problem in \cite{BDKZ2}.

Using \eqref{IBVP3-vis-1} and \eqref{IBVP-vis-5}, we find that the $O(\epsilon^{1/2})$ term in \eqref{IBVP3-BC} equals to zero. Indeed, we compute
$$
B_1\begin{pmatrix}
    \widehat{u}_0\\
    \widehat{v}_0
\end{pmatrix}
=-2(0,\Phi_1,0,\Phi_3,\dots,0,\Phi_{2N-1})\begin{pmatrix}
P_0P_0^T\widehat{u}_0\\[1mm]
0
\end{pmatrix}=-2(0,\Phi_1,0)P_0P_0^T\widehat{u}_0=0. 
$$
Note that the last equality is due to the expression of $P_0$ in \eqref{IBVP3-R1+R10R1-}. At last, we compute the $O(\epsilon)$ term in the boundary condition \eqref{IBVP3-BC}:
\begin{align}\label{IBVP3-BCep}
\epsilon B_1\begin{pmatrix}
\bar{u}_1+\widehat{u}_1\\[1mm]
\bar{v}_1+\widehat{v}_1
\end{pmatrix}(0,t)=0. 
\end{align}
By \eqref{IBVP3-v1}, \eqref{IBVP3-vis-3} and \eqref{IBVP-vis-5}, we have
\begin{align*}
(\Phi_3,0,\Phi_5,\dots,0,\Phi_{2N-1})(\bar{v}_1+\widehat{v}_1) 
=&- (\Phi_3,0,\Phi_5,\dots,0,\Phi_{2N-1})A_{12}^T(\partial_x\bar{u}_{0}+\partial_z\widehat{u}_0)\\[1mm]
=&-(0,0,\alpha_3\Phi_3)(\partial_x\bar{u}_{0}+\partial_z\widehat{u}_0)\\[1mm]
=&-(0,0,\alpha_3\Phi_3)(\partial_x\bar{u}_{0}+P_0\partial_z(P_0^T\widehat{u}_0)).
\end{align*}
Furthermore, we use the expression of $P_0$ to compute
\begin{align}\label{IBVP3-Final}
(\Phi_3,0,\Phi_5,\dots,0,\Phi_{2N-1})(\bar{v}_1+\widehat{v}_1) 
=-(0,0,\alpha_3\Phi_3)\partial_x\bar{u}_{0} +\frac{\alpha_1\alpha_3}{\lambda}\Phi_3\partial_z(P_0^T\widehat{u}_0).
\end{align}
On the other hand, using \eqref{IBVP3-vis-6} and the decomposition $\bar{u}_1=P_{+}\beta_{1+} + P_{0}\beta_{10} + P_{-}\beta_{1-}$, we have 
\begin{align*}
(0,\Phi_1,0)(\bar{u}_1+\widehat{u}_1)&=(0,\Phi_1,0)(P_+\beta_{1+}+P_0\beta_{10}+P_-\beta_{1-}+P_0P_0^T\widehat{u}_1+P_1P_1^T\widehat{u}_1)\\[1mm]
&=(0,\Phi_1,0)(P_+\beta_{1+}+P_-\beta_{1-}+P_1P_1^T\widehat{u}_1)\\[1mm]
&=(0,\Phi_1,0)(P_+\beta_{1+}+P_-\beta_{1-}-P_1\Lambda_1^{-1}P_1^TA_{12}\widehat{v}_1)\\[1mm]
&=(0,\Phi_1,0)(P_+\beta_{1+}+P_-\beta_{1-}+P_1\Lambda_1^{-1}P_1^TA_{12}A_{12}^TP_0\partial_z(P_0^T\widehat{u}_0)).
\end{align*} 
Using the expression of $P_+$ and $P_-$, we finally get
$$
(0,\Phi_1,0)(\bar{u}_1+\widehat{u}_1)=\left(\frac{1}{\sqrt{2}}\beta_{1+}-\frac{1}{\sqrt{2}}\beta_{1-}-\frac{\alpha_1\alpha_2\alpha_3^2}{\lambda^3}\partial_z(P_0^T\widehat{u}_0)\right)\Phi_1.
$$
Substituting this with \eqref{IBVP3-Final} into the boundary condition \eqref{IBVP3-BCep}, we have 
\begin{align}\label{IBVP3-ReducedBC}
    \begin{pmatrix}
        \Phi_1&\Phi_3
    \end{pmatrix}
    \begin{pmatrix}
        \frac{1}{\sqrt{2}} & \frac{-\alpha_1\alpha_2\alpha_3^2}{\lambda^3}\\[2mm]
        0 & \frac{\alpha_1\alpha_3}{\lambda}
    \end{pmatrix}
    \begin{pmatrix}
        \beta_{1+}\\[2mm]
        \partial_z(P_0^T\widehat{u}_0)
    \end{pmatrix}=
    \begin{pmatrix}
        \Phi_1 & \Phi_3
    \end{pmatrix}
    \begin{pmatrix}
    \frac{1}{\sqrt{2}}\beta_{1-}\\[2mm]
    (0,0,\alpha_3)\partial_x\bar{u}_0
    \end{pmatrix}.
\end{align}
Clearly, from this equation, we can uniquely solve $\beta_{1+}$ and $\partial_z(P_0^T\widehat{u}_0)$ in terms of $\beta_{1-}$ and $\partial_x\bar{u}_0$.
In this way, we complete the derivation of reduced boundary condition for \eqref{IBVP3-u1} and \eqref{IBVP3-vis-parabolic}.

\subsection{Second-type collision term $Q_2$: IBVP 4}
In this subsection, we denote $U=U^{(2)}$ and consider the IBVP
\begin{align}\label{IBVP4}
\left\{\begin{array}{l}
U_t+AU_x=Q_2U/\epsilon,\\[2mm]
B_2U(0,t)=0.
\end{array}\right.
\end{align} 
The notations of $A_{11}$, $A_{12}$ and $A_{22}$ are the same as those in \eqref{IBVP3-A11} and \eqref{IBVP3-A22}.
Consider the asymptotic solution
$$
U_\epsilon = \sum_{k=0}^3\epsilon^{k/2}\bar{U}_k (x,t)+\sum_{k=0}^3\epsilon^{k/2}\widehat{U}_k\left(z,t\right)+\sum_{k=0}^3\epsilon^{k/2}\widetilde{U}_k\left(y,t\right),\qquad z=\frac{x}{\sqrt{\epsilon}},\quad y=\frac{x}{\epsilon}.
$$
Here $\bar{U}_k$ is the outer solution, $\widehat{U}_k$ and $\widetilde{U}_k$ represent the boundary-layer correction term. 
Thus, for the second boundary value problem with collision term $Q_2$ we expect kinetic and viscous boundary layers to appear.

By the matching principle, we have 
$$
\widehat{U}_k(\infty,t)=0,\qquad \widetilde{U}_k(\infty,t)=0.
$$
Denote $\bar{u}_k$ by the first three components of $\bar{U}_k$ and $\bar{v}_k$ the last $(2N-3)$ components of $\bar{U}_k$. For the outer equation, we match the coefficients of $\epsilon$ in the equation 
$$
\sum_{k=0}^3 \epsilon^{k/2}
\partial_t\begin{pmatrix}
    \bar{u}_k\\[1mm]
    \bar{v}_k
\end{pmatrix}
+\sum_{k=0}^3 \epsilon^{k/2}\begin{pmatrix}
    A_{11} & A_{12}\\[1mm]
    A_{12}^T & A_{22}
\end{pmatrix}
\partial_x\begin{pmatrix}
    \bar{u}_k\\[1mm]
    \bar{v}_k
\end{pmatrix}=
\sum_{k=0}^3 \epsilon^{(k-2)/2}\begin{pmatrix}
    0\\[1mm]
    -\bar{v}_k
\end{pmatrix}
$$
to obtain:
\begin{align}
&O(\epsilon^{(k-2)/2}):\quad \bar{v}_{k} = -\partial_t\bar{v}_{k-2} - A_{12}^T\partial_x \bar{u}_{k-2} - A_{22}\partial_x\bar{v}_{k-2}\label{outer-neq-p-v},\qquad 0\leq k\leq 3,\\[2mm]
&O(\epsilon^{k/2}):~\quad\quad \partial_t\bar{u}_k + A_{11} \partial_x \bar{u}_k = -A_{12}\partial_x\bar{v}_k,\qquad 0\leq k\leq 3.\label{outer-neq-p-u} 
\end{align}
Note that if the index $k$ of $\bar{u}_k$ or $\bar{v}_k$ is negative, the corresponding term equals to zero. Given the coefficients up to order $k-1$, we solve the $k$-th order term $\bar{v}_k$ by \eqref{outer-neq-p-v}. Having this, we notice that the equation \eqref{outer-neq-p-u} is a hyperbolic system with source term. 
Since this equation is given in the half space $x>0$, proper boundary conditions should be prescribed. To this end, we introduce the characteristic variable $\beta_{k+}$, $\beta_{k0}$ and $\beta_{k-}$ for \eqref{outer-neq-p-u} such that
$$
\bar{u}_k=P_{+}\beta_{k+} + P_{0}\beta_{k0} + P_{-}\beta_{k-}.
$$
where $P_{+}$, $P_{0}$ and $P_{-}$ are given in \eqref{IBVP3-R1+R10R1-}. Due to the classical theory of hyperbolic equation, one boundary condition should be given for $\beta_{k+}$.

Next we discuss the equations for viscous boundary layers where the spatial variable is $z=x/\sqrt{\epsilon}$.
Denote $\widehat{u}_k$ by the first three components of $\widehat{U}_k$ and $\widehat{v}_k$ the last $(2N-3)$ components of $\widehat{U}_k$. We match the coefficients of $\epsilon$ in the equation 
$$
\sum_{k=0}^3 \epsilon^{k/2}
\partial_t\begin{pmatrix}
    \widehat{u}_k\\[1mm]
    \widehat{v}_k
\end{pmatrix}
+\sum_{k=0}^3 \epsilon^{(k-1)/2}\begin{pmatrix}
    A_{11} & A_{12}\\[1mm]
    A_{12}^T & A_{22}
\end{pmatrix}
\partial_z\begin{pmatrix}
    \widehat{u}_k\\[1mm]
    \widehat{v}_k
\end{pmatrix}=
\sum_{k=0}^3 \epsilon^{(k-2)/2}\begin{pmatrix}
    0\\[1mm]
    -\widehat{v}_k
\end{pmatrix}
$$
to obtain:
\begin{align}
&O(\epsilon^{(k-2)/2}):\quad \widehat{v}_{k} = -\partial_t\widehat{v}_{k-2} - A_{12}^T\partial_z \widehat{u}_{k-1} - A_{22}\partial_z\widehat{v}_{k-1} ,\qquad 0\leq k\leq 3,\label{vis-k-v}\\[2mm]
&O(\epsilon^{(k-1)/2}):\quad A_{11} \partial_z \widehat{u}_k = - \partial_t\widehat{u}_{k-1} -A_{12}\partial_z\widehat{v}_k,\qquad 0\leq k\leq 3. \label{vis-k-u}
\end{align}
Now we show the steps to determine the viscous layer terms $\widehat{u}_k$ and $\widehat{v}_k$ with given lower order terms. At first, we get $\widehat{v}_k$ from \eqref{vis-k-v}. Recall that $\Lambda_1=\text{diag}(\lambda,-\lambda)$ and $P_1=(P_+,P_-)$ are the nonzero eigenvalues and associated eigenvectors of $A_{11}$. It follows from \eqref{vis-k-u} that
\begin{align*}
P_1^T\widehat{u}_k=-\Lambda^{-1} P_1^T(A_{12}\widehat{v}_k-\int_z^{\infty}\partial_t\widehat{u}_{k-1}).
\end{align*}
Then it suffices to solve $P_0^T\widehat{u}_k$. To this end, we multiply $P_0^T$ on the left of \eqref{vis-k-u} to obtain
\begin{align*}
\partial_t(P_0^T\widehat{u}_{k}) + P_0^TA_{12}\partial_z\widehat{v}_{k+1} =0.
\end{align*}
By using \eqref{vis-k-v} and the decomposition $\widehat{u}_{k}=P_0P_0^T\widehat{u}_{k}+P_1P_1^T\widehat{u}_{k}$, we have 
\begin{align}\label{vis-k-parabolic}
\partial_t(P_0^T\widehat{u}_{k}) - (P_0^TA_{12}A_{12}^TP_0)\partial_{zz}(P_0^T\widehat{u}_{k}) = \widehat{f}_k.
\end{align}
for $0\leq k\leq 2$ with 
$$
\widehat{f}_k=P_0^TA_{12}\partial_z[\partial_t \widehat{v}_{k-1}+A_{12}^TP_1\partial_z(P_1^T\widehat{u}_{k})+A_{22}\partial_z\widehat{v}_{k}].
$$
With proper initial and boundary conditions, we can solve the parabolic partial differential equation \eqref{vis-k-parabolic} to get $P_0^T\widehat{u}_{k}$. 
\begin{remark}
    From equations \eqref{vis-k-v} and \eqref{vis-k-u}, we can determine $\widehat{v}_k$ and $P_1^T\widehat{u}_k$ for  $0\leq k\leq 3$, but the term $P_0^T\widehat{u}_k$ only for $0\leq k\leq 2$. However, we notice that $\widehat{f}_k$ in the equation \eqref{vis-k-parabolic} is well-defined for $k=3$. Thus we can still solve $P_0^T\widehat{u}_3$  from \eqref{vis-k-parabolic} if proper initial and boundary conditions are prescribed. 
\end{remark}

For the boundary-layer term $\widetilde{U}_k$, we match the coefficients of $\epsilon$ in the equation 
$$
\sum_{k=0}^3\epsilon^{k/2}\partial_t\widetilde{U}_{k}+\sum_{k=0}^3\epsilon^{(k-2)/2}A\partial_y
\widetilde{U}_k
=\sum_{k=0}^3\epsilon^{(k-2)/2}Q_2\widetilde{U}_k
$$
to obtain equations for $\widetilde{U}_k$:
\begin{align}\label{Knudsen-k}
O(\epsilon^{(k-2)/2})~~0\leq k\leq 3\qquad
A\partial_y
\widetilde{U}_k
=Q_2\widetilde{U}_k-\partial_t\widetilde{U}_{k-2}.
\end{align}
For each $k$, the boundary layer term $\widetilde{U}_k$ satisfies the same ordinary differential system 
$$
A\partial_y
\widetilde{U}_k
=Q_2\widetilde{U}_k+l.o.t.
$$
Here we briefly write the lower-order term as l.o.t. and omit the explicit expressions. 
For the discussion of this equation, we write the component-wise expression 
$$
\widetilde{U}_k = \begin{pmatrix}
    \widetilde{u}_k\\[1mm]
    \widetilde{v}_k
\end{pmatrix},\quad \widetilde{u}_k = (\widetilde{U}_{k,0},\widetilde{U}_{k,1},\widetilde{U}_{k,2})^T,\quad \widetilde{v}_k=(\widetilde{U}_{k,3},\widetilde{U}_{k,4},\dots,\widetilde{U}_{k,2N-1})^T.
$$ 
According to the expression of $A$ and $Q_2$, we have 
\begin{align*}
    \widetilde{U}_{k,1} &= l.o.t.,\qquad \alpha_1\widetilde{U}_{k,0} + \alpha_2\widetilde{U}_{k,2} = l.o.t., 
    \qquad \widetilde{U}_{k,3} = l.o.t..
\end{align*}
Due to the last equality, we also derive $\alpha_4\widetilde{U}_{k,4}+\alpha_3\widetilde{U}_{k,2}=l.o.t.$. To determine $\widetilde{v}_{k,H} = (\widetilde{U}_{k,4},\dots,\widetilde{U}_{k,2N-1})$, we have 
\begin{align}\label{eq-n-U0neqtilde-1}
\widetilde{A}_H\partial_y\widetilde{v}_{k,H} = -\widetilde{v}_{k,H},\qquad \widetilde{A}_H=\begin{pmatrix}
   0 & \alpha_{5} &  &  \\
   \alpha_5 & 0 & \ddots &  \\
    & \ddots & \ddots & \alpha_{2N-1} \\[1mm]
    &  & \alpha_{2N-1} & 0 
\end{pmatrix}.
\end{align}
For the bounded solution, the initial data should satisfy 
$$
\widetilde{v}_{k,H}(0)= \widetilde{R}_{-} \gamma_{k-}.
$$
where $\widetilde{R}_{-}$ represents the eigenvectors of $-\widetilde{A}_H^{-1}$ associated with negative eigenvalues. To conclude, we find that 
\begin{align}
&\widetilde{u}_k = 
    \widetilde{N} \widetilde{v}_{k,H} +l.o.t.= 
    \widetilde{N} \widetilde{R}_{-}\gamma_{k-}+l.o.t.~~\text{with}~~
    \widetilde{N}=\begin{pmatrix}
    \frac{\alpha_2\alpha_4}{\alpha_1\alpha_3} & 0 & \cdots & 0\\[2mm]
    0 & 0 & \cdots& 0\\[1mm]
    \frac{-\alpha_4}{\alpha_3} & 0 & \cdots& 0
\end{pmatrix} \label{knudsen-k1}\\
&\widetilde{v}_k = \begin{pmatrix}
    0 \\[1mm]
    \widetilde{v}_{k,H}
\end{pmatrix}+l.o.t.
=\begin{pmatrix}
    0 \\[1mm]
    \widetilde{R}_{-}\gamma_{k-}
\end{pmatrix}+l.o.t.. \label{knudsen-k2}
\end{align}

We substitute the asymptotic solution into the boundary condition and match the coefficients of $\epsilon$ to get
$$
O(\epsilon^{k/2}),~0\leq k\leq 3:\qquad B_2\begin{pmatrix}
    \bar{u}_k+\widehat{u}_k+\widetilde{u}_k\\[1mm]
    \bar{v}_k+\widehat{v}_k+\widetilde{v}_k
\end{pmatrix}(0,t)=0.
$$
Recall the relations $\bar{u}_k=P_{+}\beta_{k+} + P_{0}\beta_{k0} + P_{-}\beta_{k-}$ and $\widehat{u}_k=P_1(P_1^T\widehat{u}_k)+P_0(P_0^T\widehat{u}_k)$.
For the outer solution $\bar{U}_k(0,t)$ and the viscous boundary layer $\widehat{U}_k(0,t)$, the unknowns are $\beta_{k+}$ and $P_0^T\widehat{u}_k$ while other terms are regarded as given functions. Furthermore, we use \eqref{knudsen-k1} and \eqref{knudsen-k2} to write
\begin{align}\label{eq-problem4-ODE}
&B_2
\begin{pmatrix}
P_{+} & P_{0} & \widetilde{N}\widetilde{R}_{-} \\[1mm]
0 & 0 & 0 \\[1mm]
0 & 0 & \widetilde{R}_{-} 
\end{pmatrix}
\begin{pmatrix}
\beta_{k+} \\[2mm]
P_0^T\widehat{u}_k  \\[1mm]
\gamma_{k-}
\end{pmatrix}=RHS.
\end{align}
Notice that all terms on the right-hand side (RHS) are known and we omit the expression in discussing the solvability of this algebraic equation.

For the case $n\geq 3$, we know that the matrix $B_2$ satisfies the strictly dissipative condition which implies the generalized Kreiss condition \cite{Yong-99-IUMJ,ZY-22}. In this case, the result in \cite{ZY-22} ensures the solvability (see Lemma 4.6 in \cite{ZY-22}). For the case $n=2$, we use the expression for $B_2$ to write 
$$
B_2 = (I_N,(n-1)I_N)V^T = 2\begin{pmatrix}
    \Phi_0,0,\Phi_2,0,\cdots,\Phi_{2N-2},0
\end{pmatrix}.
$$
Then it follows that
\begin{align*}
&B_2
\begin{pmatrix}
P_{+} & P_{0} & \widetilde{N}\widetilde{R}_{-} \\[1mm]
0 & 0 & 0 \\[1mm]
0 & 0 & \widetilde{R}_{-} 
\end{pmatrix}=2\begin{pmatrix}
    \Phi_0,\Phi_2,\cdots,\Phi_{2N-2}
\end{pmatrix}
\begin{pmatrix}
\frac{\alpha_1}{\sqrt{2}\lambda} & \frac{\alpha_2}{\lambda} & * \\[2mm]
\frac{\alpha_2}{\sqrt{2}\lambda} & \frac{-\alpha_1}{\lambda} & * \\[2mm]
0 & 0 & \widetilde{R}^{e}_{-}
\end{pmatrix}.
\end{align*}
Here $\widetilde{R}^{e}_{-}$ consists of all the even-numbered rows of the original $\widetilde{R}_{-}$. It is not difficult to verify that the matrix $(\Phi_0,\Phi_2,\cdots,\Phi_{2N-2})$ is invertible. Moreover, we compute
$$
\det\begin{pmatrix}
\frac{\alpha_1}{\sqrt{2}\lambda} & \frac{\alpha_2}{\lambda}\\[2mm]
\frac{\alpha_2}{\sqrt{2}\lambda} & \frac{-\alpha_1}{\lambda}
\end{pmatrix}=-\frac{\alpha_1^2+\alpha_2^2}{\sqrt{2}\lambda^2}=-\frac{1}{\sqrt{2}}\neq 0.
$$ 
At last, we show that the matrix $\widetilde{R}^{e}_{-}$ is invertible. Indeed, define the transform matrix $T$ such that the moments $\widetilde{v}_{k,H}$ can be reordered according to the even-odd partition:
$$
T~\widetilde{v}_{k,H}= \begin{pmatrix}
    \widetilde{v}_{k,e}\\[1mm]
    \widetilde{v}_{k,o}
\end{pmatrix},\quad \widetilde{v}_{k,e}=(\widetilde{v}_{k,4},\dots,\widetilde{v}_{k,2N-2}),\quad \widetilde{v}_{k,o}=(\widetilde{v}_{k,5},\dots,\widetilde{v}_{k,2N-1}).
$$
By the expression of $\widetilde{A}_H$, we can write 
$$
T\widetilde{A}_HT^{-1} = \begin{pmatrix}
    0 & A_c\\[1mm]
    A_c^T & 0
\end{pmatrix},\qquad A_c=\begin{pmatrix}
\alpha_5 & 0 &  \cdots & 0\\
\alpha_6 & \alpha_7 & \ddots & \vdots\\
 & \ddots & \ddots & 0\\
 & & \alpha_{2N-2} & \alpha_{2N-1}
\end{pmatrix}
$$
with $A_c$ an invertible matrix. There exist $R_e$ and $R_o$ such that 
$$
\begin{pmatrix}
    0 & A_c\\[1mm]
    A_c^T & 0
\end{pmatrix}
\begin{pmatrix}
    R_e & R_e\\[1mm]
    R_o & -R_o
\end{pmatrix}=\begin{pmatrix}
    R_e & R_e\\[1mm]
    R_o & -R_o
\end{pmatrix}
\begin{pmatrix}
    \Lambda_c & 0\\[1mm]
    0 & -\Lambda_c
\end{pmatrix}
$$
and $R_e^TR_e=R_o^TR_o=\frac{1}{2}I$. Here $\Lambda_c$ is a positive definite diagonal matrix representing the positive eigenvalues of $\widetilde{A}_H$. Then we can choose the eigenvectors $\widetilde{R}_{-}$ for $\widetilde{A}_H$ by
$$
\widetilde{R}_{-}=T^{-1}\begin{pmatrix}
    R_e\\
    -R_o
\end{pmatrix}.
$$ 
Therefore $\widetilde{R}^{e}_{-}=(I,~0)T\widetilde{R}_{-}=R_e$ is invertible. Finally, we prove the solvability of \eqref{eq-problem4-ODE} in the case $n=2$.

\begin{remark}
	\label{rem2}
Similar to that in Remark \ref{remark5-m}, we explain how the
solvability result of \eqref{eq-problem4-ODE} justifies the derivation in \cite{BDKZ2}. In \cite{BDKZ2} (see Section 5 therein), the coupling conditions are derived from the fact that the matrix $RS^{-1}T$ has full-rank, where
$$
T = \begin{pmatrix}
    T_{11} & T_{12}\\[1mm]
    0 & R_2^+
\end{pmatrix},\qquad 
T_{11}=
\begin{pmatrix}
        0 & 0 & \frac{1}{\sqrt{2}}\\[1mm]
        0 & \frac{1}{\sqrt{2}} & 0\\[1mm]
        \frac{1}{2} & 0 & -\frac{1}{2}\\[1mm]
        0 & 0 & 0
        \end{pmatrix},\quad 
        T_{12}=
\begin{pmatrix}
        \frac{2\sqrt{2}}{\sqrt{3}}e_1^TR_2^+\\[1mm]
        0 \\[1mm]
        -\frac{2}{\sqrt{3}}e_1^TR_2^+\\[1mm]
        0 
        \end{pmatrix}.
$$
To unify the notation,  the matrix $RS^{-1}$  corresponds to the matrix $B_2$ in the present work and $R_2^+$ corresponds to the matrix $\widetilde{R}_{-}$ in this work. In addition, we compute
    $$
    \widetilde{N}=\begin{pmatrix}
    \frac{2\sqrt{2}}{\sqrt{3}}e_1^T \\[1mm]
    0 \\[1mm]
    \frac{-2}{\sqrt{3}} e_1^T
\end{pmatrix}.
    $$
    Therefore, we can rewrite the matrix $T$ as 
    $$
    T=\begin{pmatrix}
        \widetilde{T} & \widetilde{N}\widetilde{R}_{-}\\[1mm]
        0 & 0 \\[1mm]
        0 & \widetilde{R}_{-}
    \end{pmatrix}\quad \text{with} \quad \widetilde{T} = \begin{pmatrix}
        0 & 0 & \frac{1}{\sqrt{2}}\\[1mm]
        0 & \frac{1}{\sqrt{2}} & 0\\[1mm]
        \frac{1}{2} & 0 & -\frac{1}{2}
    \end{pmatrix}.
    $$
    Now we show that the matrix $B_2T$ has full rank. Let $P_-$ be the eigenvector associated with the negative eigenvalue such that $(P_-,P_+,P_0)$ is invertible. Then we know from the
solvability of \eqref{eq-problem4-ODE} that 
    $$
    B_2\begin{pmatrix}
P_- & P_{+} & P_{0} & \widetilde{N}\widetilde{R}_{-} \\[1mm]
0 & 0 & 0 & 0 \\[1mm]
0 & 0 & 0 & \widetilde{R}_{-} 
\end{pmatrix}
    $$
    is a full-rank $N\times (N+1)$ matrix. Thus the full-rank property of $B_2T$ follows from the relation
    $$
    B_2T = B_2\begin{pmatrix}
        \widetilde{T} & \widetilde{N}\widetilde{R}_{-}\\[1mm]
        0 & 0 \\[1mm]
        0 & \widetilde{R}_{-}
    \end{pmatrix}= B_2\begin{pmatrix}
P_- & P_{+} & P_{0} & \widetilde{N}\widetilde{R}_{-} \\[1mm]
0 & 0 & 0 & 0 \\[1mm]
0 & 0 & 0 & \widetilde{R}_{-} 
\end{pmatrix}\begin{pmatrix}
    (P_-,P_+,P_0)^{-1}\widetilde{T} & 0 \\[1mm]
    0 & I_{N-2}
\end{pmatrix}.
    $$
For further details on the numerical solutions of this boundary condition, we refer the interested reader to \cite{BDKZ2}.
\end{remark}

\section{Error Estimate}\label{Section6} 

In this section, we establish error estimates that justify the
asymptotic expansions constructed in Section~\ref{section5}. Recall
that the original edge variables $\{G^{(i)}\}_{i=1}^n$ and the
transformed variables $\{U^{(i)}\}_{i=1}^n$ are related by the linear
invertible transformation
\[
U^{(1)}=\sum_{k=1}^nG^{(k)},
\qquad
U^{(i)}=G^{(i)}-G^{(1)},\quad i=2,\ldots,n,
\]
whose inverse is
\[
G^{(1)}
=
\frac1n\left(
U^{(1)}-\sum_{k=2}^nU^{(k)}
\right),
\qquad
G^{(i)}=U^{(i)}+G^{(1)},\quad i=2,\ldots,n.
\]
Since the transformation is independent of $\epsilon$, the construction of the asymptotic expansion in the
variables $\{U^{(i)}\}_{i=1}^n$ is equivalent to its construction in
the variables $\{G^{(i)}\}_{i=1}^n$.

The main result is stated below.

\begin{theorem}[Error estimates for the network problems]
\label{thm:error-estimates}
Let \(T>0\) be fixed. For \(i=1,\ldots,n\), let \(G^{(i)}\) denote the exact solution of the discrete kinetic problem on the \(i\)-th edge, governed by \eqref{eq:GQ1} for the first-type collision operator $Q_1$ or by \eqref{eq:GQ2} for the second-type collision operator $Q_2$, and subject to the coupling condition \eqref{CCG1}--\eqref{CCG2}. Let
\(G_{\epsilon}^{(i)}\) denote the corresponding asymptotic
approximation constructed in Section~\ref{section5}.

Assume that the initial data are well prepared, sufficiently smooth,
and satisfy the compatibility conditions stated in the corresponding
subsections below. Then there exist constants $C>0$ and
$\epsilon_0>0$, independent of $\epsilon$, such that, for every
$0<\epsilon\leq\epsilon_0$, the following estimates hold:
\begin{equation}\label{eq:theorem-error}
\max_{1\leq i\leq n}\sup_{t\in[0,T]}
\left\|
G^{(i)}(\cdot,t)-G_{\epsilon}^{(i)}(\cdot,t)
\right\|_{H^1(\mathbb{R}^+)}
\leq
\begin{cases}
C\epsilon^{1/2}, & \quad \text{for } Q_1,\\[2mm]
C\epsilon^{1/4}, & \quad \text{for } Q_2.
\end{cases}
\end{equation}
Consequently, by the Sobolev embedding
$H^1(\mathbb{R}^+)\hookrightarrow L^\infty(\mathbb{R}^+)$,
\begin{align}
\max_{1\leq i\leq n}\sup_{t\in[0,T]}
\left\|
G^{(i)}(\cdot,t)-G_{\epsilon}^{(i)}(\cdot,t)
\right\|_{L^\infty(\mathbb{R}^+)}
\leq
\begin{cases}
C\epsilon^{1/2}, & \quad \text{for } Q_1,\\[2mm]
C\epsilon^{1/4}, & \quad \text{for } Q_2.
\end{cases}
\label{eq:theorem-uniform-Q2}
\end{align}
\end{theorem}

\begin{remark}
The assumptions on the initial data and compatibility
conditions required in Theorem~\ref{thm:error-estimates} are stated
explicitly at the beginning of the subsections below.
In view of the equivalence between the two sets of variables, we
formulate these assumptions in terms of
$\{U^{(i)}\}_{i=1}^n$.
\end{remark}

The rest of this section is devoted to the proof of
Theorem~\ref{thm:error-estimates}. We first establish a general energy
estimate for the error equation and then verify the required residual
estimates separately for the four transformed initial-boundary value
problems.
We start with:
\begin{lemma}\label{lemma5.1}
Consider the IBVP 
\begin{align}\label{estimate-eq1}
\left\{\begin{array}{l}
     \partial_t U + A \partial_x U = \dfrac{1}{\epsilon} Q U - \begin{pmatrix}
        0 \\
        E_1
    \end{pmatrix}-E_2, \\
    B_k U(0,t) = 0, \\[2mm]
    U(x,0) = 0.
\end{array}\right.
\end{align}

Here the unknown $U$ is a $m$-dimensional vector, $A$ is a symmetric invertible matrix, $B_k$ satisfies the dissipative condition and $Q=\text{diag}(0,-I_{r})$ with $r\leq m$. Moreover, we assume that the residual terms $E_1\in \mathbb{R}^{r}$ and $E_2\in \mathbb{R}^{m}$ belong to $H^1([0,T]\times \mathbb{R}^+)$ and satisfy
\begin{align*}
&\|\partial_tE_1(\cdot,t)\|_{L^2(\mathbb{R}^+)}+\|E_1(\cdot,t)\|_{L^2(\mathbb{R}^+)} \leq C\epsilon^{1/2+\kappa}\\[2mm]
&\|\partial_tE_2(\cdot,t)\|_{L^2(\mathbb{R}^+)}+\|E_2(\cdot,t)\|_{L^2(\mathbb{R}^+)} \leq C\epsilon^{1+\kappa}\qquad \kappa>0.
\end{align*} 
Then the solution $U$ of the problem \eqref{estimate-eq1} satisfies the estimate
\begin{align*}
    \max_{t\in[0,T]}\|U(\cdot,t)\|_{H^1(\mathbb{R}^+)}\leq C\epsilon^\kappa.
\end{align*}
\end{lemma}

\begin{proof}
For the estimate, we multiply \eqref{estimate-eq1} with $U^T$ from the left to obtain
$$
\frac{d}{dt}\left(U^TU\right)+ \left(U^TAU\right)_{x}= -\dfrac{2}{\epsilon}
\left|U_{neq}\right|^2-2\big(U_{neq}\big)^TE_1-2U^TE_{2}.
$$
Here $U_{neq}$ represents the last $r$ components of $U$. From the last equation, we use the Cauchy inequality and derive
\begin{align*}
\frac{d}{dt}\left(U^TU\right)+ \left(U^TAU\right)_{x}= & -\dfrac{2}{\epsilon}
\left|U_{neq}\right|^2-2\big(U_{neq}\big)^TE_1 -2U^TE_{2}\\[2mm]
\leq &  -\dfrac{1}{\epsilon}
\left|U_{neq}\right|^2+ \epsilon|E_1|^2 + |U|^2 + |E_{2}|^2.
\end{align*}
Integrating the last inequality over $x\in[0,+\infty)$ yields
\begin{align}
&\frac{d}{dt}\|U(\cdot,t)\|_{L^2(\mathbb{R}^+)}^2 - U^TAU|_{x=0}\nonumber \\[2mm]
\leq &~~\epsilon \|E_1(\cdot,t)\|_{L^2(\mathbb{R}^+)}^2 +  \|U(\cdot,t)\|_{L^2(\mathbb{R}^+)}^2 + \|E_2(\cdot,t)\|_{L^2(\mathbb{R}^+)}^2.\label{5.14}
\end{align}
For the boundary value $U(0,t)$ satisfying $B_1U(0,t)=0$ or $B_2U(0,t)=0$, we use Lemma \ref{lemma2.2} and Lemma \ref{lemma2.3} to obtain
$$
U^TAU|_{x=0}\leq 0.
$$ 
Applying Gronwall's inequality to \eqref{5.14}, we have $\forall t\in[0,T]$:
\begin{align*}
&\|U(\cdot,t)\|_{L^2(\mathbb{R}^+)}^2 
\leq C\bigg(\epsilon\|E_1\|^2_{L^2([0,T] \times \mathbb{R}^+ )} + \|E_2\|^2_{L^2([0,T] \times \mathbb{R}^+)} \bigg)\leq C\epsilon^{2+2\kappa}. 
\end{align*}
Next we take $\partial_t$ in the equation \eqref{estimate-eq1} to obtain
\begin{align}\label{estimate-eq2}
\left\{\begin{array}{l}
     \partial_t U_t + A \partial_x U_t = \dfrac{1}{\epsilon} Q U_t - \begin{pmatrix}
        0 \\
        \partial_tE_1
    \end{pmatrix}-\partial_tE_2, \\
    B_k U_t(0,t) = 0, \\[2mm]
    U_t(x,0) = U_{t,0}(x).
\end{array}\right.
\end{align}
Considering $t=0$ in \eqref{estimate-eq1}, we know that 
$$
\|U_{t,0}\|_{L^2(\mathbb{R}^+)}\leq C(\|E_1(\cdot,0)\|_{L^2(\mathbb{R}^+)}+\|E_2(\cdot,0)\|_{L^2(\mathbb{R}^+)})\leq C\epsilon^{1/2+\kappa}.
$$ Following the $L^2$ estimate for $U$, we have 
\begin{align*}
\|U_t(\cdot,t)\|_{L^2(\mathbb{R}^+)}^2 
\leq&~ C\bigg(\|U_{t,0}\|_{L^2(\mathbb{R}^+)}^2+ \epsilon\|\partial_tE_1\|^2_{L^2([0,T] \times \mathbb{R}^+ )} + \|\partial_tE_2\|^2_{L^2([0,T] \times \mathbb{R}^+)} \bigg)\\[1mm]
\leq&~ C\epsilon^{1+2\kappa}.
\end{align*}
At last, we use the equation \eqref{estimate-eq1} to get 
\begin{align*}
\|U_x(\cdot,t)\|_{L^2(\mathbb{R}^+)} \leq&~ C\Big(\|U_t(\cdot,t)\|_{L^2(\mathbb{R}^+)} +\epsilon^{-1} \|U(\cdot,t)\|_{L^2(\mathbb{R}^+)}
+ \epsilon^{1/2+\kappa}\Big)\\[1mm]
\leq &~C\epsilon^{\kappa}.
\end{align*}
This completes the proof of the lemma.
\end{proof}

In the following subsections, we verify that the asymptotic solutions constructed in Section \ref{section5} satisfy the assumptions of Lemma \ref{lemma5.1}.

\subsection{First-type collision term $Q_1$, IBVP 1} 
At first, we prescribe some assumptions for the initial data such that the solutions are smooth enough and the initial layers are avoided. Specifically, we assume that 
\begin{itemize}
    \item[(i)] The initial data for the original relaxation problem are given in equilibrium which means 
$$
U(x,0)=\bar{U}_0(x,0)+\epsilon \bar{U}_1(x,0)
= \begin{pmatrix}
    \bar{u}_0\\
    \bar{v}_0
\end{pmatrix}(x,0)+\epsilon\begin{pmatrix}
    \bar{u}_1\\
    \bar{v}_1
\end{pmatrix}(x,0)
$$
where $\bar{v}_0(x,0)$ and $\bar{v}_1(x,0)$ satisfy the relation \eqref{IBVP1-bar-v0} and \eqref{IBVP1-bar-v1}.
\item[(ii)] The initial data $\bar{u}_0(x,0)$ and $\bar{u}_1(x,0)$ are smooth enough and compatible with the boundary condition \eqref{reducedBC1} up to a certain order such that $\bar{U}_0\in CH_T^3$ and $\bar{U}_1\in CH_T^2$ with the definition
$$
CH_T^p = \cap_{k\leq p} C^k([0,T];H^{p-k}(\mathbb{R}^+)).
$$
\end{itemize} 
By assumption (i), the initial data for the difference between the exact solution and the asymptotic solution should be zero. Namely, $U_{err}(x,0)=U(x,0)-U_{\epsilon}(x,0)=0$. Moreover, thanks to assumptions (i) and (ii), it is not difficult to check that the initial data $U(x,0)=\bar{U}_0(x,0)+\epsilon \bar{U}_1(x,0)$ are compatible with the original boundary condition $B_1U(0,t)=0$ which means 
$
B_1\bar{U}_0(0,0)=B_1\bar{U}_1(0,0)=0.
$

Denote by $\mathcal{L}_1$ the differential operator
\begin{equation*}
	\mathcal{L}_1(U):=\partial_tU+A\partial_{x}U-Q_1U/\epsilon.
\end{equation*}
Clearly, the solution to the original relaxation system satisfies $\mathcal{L}_1(U)=0$. If we substitute the asymptotic solution into the operator, a residual term may occur.
By \eqref{IBVP1-bar-v0}---\eqref{IBVP1-bar-u1}, we have
\begin{align}
\mathcal{L}_1(U_\epsilon)=\mathcal{L}_1(\bar{U}_0+\epsilon \bar{U}_1)=&\begin{pmatrix}
0\\
E_1
\end{pmatrix},\quad
E_1=\epsilon 
(\partial_{t}\bar{v}_{1} +A_{12}^T\partial_x \bar{u}_1+A_{22}\partial_x \bar{v}_1).\label{residual-1}
\end{align}
Then we know that $U_{err}=U-U_{\epsilon}$ satisfies the equation in \eqref{estimate-eq1} and the residual term $E_1$ satisfies the assumption in Lemma \ref{lemma5.1} with $\kappa=1/2$. At last, we recall the boundary condition and see that $BU_{err}(0,t)=BU(0,t)-BU_{\epsilon}(0,t)=0$. By exploiting the error estimate in Lemma \ref{lemma5.1}, we have 
$$
\max_{t\in[0,T]}\|U_{err}(\cdot,t)\|_{H^1(\mathbb{R}^+)}\leq C\epsilon^{1/2}.
$$

\subsection{First-type collision term $Q_1$, IBVP 2} The discussion for the outer solution is similar. We assume that
\begin{itemize}
    \item[(i)] The initial data for the original relaxation problem are given in equilibrium which means 
$$
U(x,0)=\bar{U}_0(x,0)+\epsilon \bar{U}_1(x,0)
= \begin{pmatrix}
    \bar{u}_0\\
    \bar{v}_0
\end{pmatrix}(x,0)+\epsilon\begin{pmatrix}
    \bar{u}_1\\
    \bar{v}_1
\end{pmatrix}(x,0)
$$
where $\bar{v}_0(x,0)$ and $\bar{v}_1(x,0)$ satisfy the relation \eqref{IBVP1-bar-v0} and \eqref{IBVP1-bar-v1}.
\item[(ii)] The initial data $\bar{u}_0(x,0)$ and $\bar{u}_{1}(x,0)$ are smooth enough and compatible with the reduced boundary condition up to a certain order such that $\bar{U}_0\in CH_T^3$ and $\bar{U}_1\in CH_T^2$. 
\item[(iii)] The initial data $U(x,0)=\bar{U}_0(x,0)+\epsilon \bar{U}_1(x,0)$ are compatible with the original boundary condition $B_2U(0,t)=0$ which means 
$
B_2\bar{U}_0(0,0)=B_2\bar{U}_1(0,0)=0.
$
\end{itemize}
Recall 
\eqref{IBVP2-U0tilde} and \eqref{IBVP2-U1tilde} that 
\begin{align}
\mathcal{L}_1(U_\epsilon)=\mathcal{L}_1(\bar{U}_0+\epsilon \bar{U}_1+\widetilde{U}_0+\epsilon \widetilde{U}_1)=&\begin{pmatrix}
0\\
E_1
\end{pmatrix}+E_2\label{residual-2}
\end{align}
with $E_1$ given in \eqref{residual-1} and $E_2=\epsilon \partial_t\widetilde{U}_1$. Based on assumption (ii), we know that the outer solutions are smooth. According to the classical existence theory \cite{BenzoniSerre}, we also have 
$\bar{U}_0(0,t)\in H^3(0,T)$ and $\bar{U}_1(0,t)\in H^2(0,T)$.
Then due to the boundary condition \eqref{eq:BClinear}, we get
$\widetilde{U}_0(0,t)\in H^3(0,T)$. Similarly, we have $\widetilde{U}_1(0,t)\in H^2(0,T)$. Clearly, the solutions $\widetilde{U}_0$ and $\widetilde{U}_1$ of the ODE system decay exponentially fast with respect to $x\rightarrow \infty$ and thereby admit the regularity in Lemma \ref{lemma5.1}. We observe that 
$$
\|E_2\|_{L^2([0,T]\times \mathbb{R}^+)}^2 = \int_0^T \int_0^\infty \epsilon^2\left|\widetilde{U}_1\left(\frac{x}{\epsilon},t\right)\right|^2dxdt = \epsilon^3 \int_0^T \int_0^\infty \left|\widetilde{U}_1\left(y,t\right)\right|^2dydt\leq C\epsilon^{3}.
$$
A similar argument holds for $\|\partial_tE_2\|_{L^2([0,T]\times \mathbb{R}^+)}\leq C\epsilon^{3/2}$. Thus we may take $\kappa=1/2$ in Lemma \ref{lemma5.1}. 

Due to assumption (iii) together with the uniqueness of \eqref{eq:BClinear}, we have $\widetilde{U}_0(0,0)=0$. Similarly, we also obtain $\widetilde{U}_1(0,0)=0$. This means $\widetilde{U}_0(x,0)= \widetilde{U}_1(x,0)\equiv 0$ and thereby $U_{err}(x,0)=U(x,0)-U_\epsilon(x,0)=0$. At last, we conclude from \eqref{IBVP2-BC} that the error $U_{err}=U-U_\epsilon$ satisfies the boundary condition $BU_{err}(0,t)=0$. Using Lemma \ref{lemma5.1}, we have 
$$
\max_{t\in[0,T]}\|U_{err}(\cdot,t)\|_{H^1(\mathbb{R}^+)}\leq C\epsilon^{1/2}.
$$

\subsection{Conclusion for the first-type collision term $Q_1$} 
Now we recall the setting for the multi-edges problem. The IBVP 1 corresponds to the term $U^{(1)}=\sum_{k=1}^nG^{(k)}$ and the IBVP 2 corresponds to the problems for $U^{(i)}=G^{(i)}-G^{(1)}$ with $i\geq 2$. Combining the results above, we obtain
for any $t\in [0,T]$ and $1\leq i\leq n$:
\begin{align*}
\|U^{(i)}(\cdot,t)-U^{(i)}_{\epsilon}(\cdot,t)\|_{L^\infty(\mathbb{R}^+)}\leq C
\|U^{(i)}(\cdot,t)-U^{(i)}_{\epsilon}(\cdot,t)\|_{H^1(\mathbb{R}^+)}\leq C\epsilon^{1/2}.
\end{align*}
Then we denote 
\begin{align*}
    G^{(1)}_\epsilon = \frac{1}{n}\Big(U^{(1)}_{\epsilon}-\sum_{k=2}^nU^{(k)}_{\epsilon}\Big),\qquad 
    G^{(i)}_\epsilon = U^{(i)}_\epsilon + G^{(1)}_\epsilon,\quad 2\leq i\leq n.
\end{align*}
By computation, we have 
\begin{align*}
&\|G^{(1)}-G^{(1)}_{\epsilon}\|_{L^\infty(\mathbb{R}^+)} = \frac{1}{n}\|U^{(1)}-U^{(1)}_{\epsilon} - \sum_{k=2}^n\Big(U^{(k)}-U^{(k)}_{\epsilon}\Big)\|_{L^\infty(\mathbb{R}^+)}\leq C\epsilon^{1/2}\\[2mm] 
&\|G^{(i)}-G^{(i)}_{\epsilon}\|_{L^\infty(\mathbb{R}^+)} = \|U^{(i)}-U^{(i)}_{\epsilon} + G^{(1)}-G^{(1)}_{\epsilon} \|_{L^\infty(\mathbb{R}^+)}\leq C\epsilon^{1/2}.
\end{align*}
As $\epsilon$ goes to zero, the asymptotic solution converges to the exact solution in $L^\infty$ norm.

\subsection{Second-type collision term $Q_2$, IBVP 3} 
We assume that 
\begin{itemize}
    \item[(i)] The initial data for the original relaxation problem are given in equilibrium which means 
$$
U(x,0)=\bar{U}_0(x,0)+\epsilon \bar{U}_1(x,0) 
= \begin{pmatrix}
    \bar{u}_0\\
    \bar{v}_0
\end{pmatrix}(x,0)+\epsilon\begin{pmatrix}
    \bar{u}_1\\
    \bar{v}_1
\end{pmatrix}(x,0)
$$
where $\bar{v}_0(x,0)$ and $\bar{v}_1(x,0)$ satisfy the relation \eqref{IBVP3-v0} and \eqref{IBVP3-v1}.
\item[(ii)] The initial data $\bar{u}_{0}(x,0)$ and $\bar{u}_{1}(x,0)$ are smooth enough and compatible with the reduced boundary condition up to a certain order such that $\bar{U}_0\in CH_T^3$ and $\bar{U}_1\in CH_T^2$. 
\item[(iii)] The initial data $U(x,0)=\bar{U}_0(x,0)+\epsilon \bar{U}_1(x,0)$ are compatible with the original boundary condition $B_1U(0,t)=0$ which means 
$
B_1\bar{U}_0(0,0)=B_1\bar{U}_1(0,0)=0.
$
\end{itemize}
Denote by $\mathcal{L}_2$ the differential operator
\begin{equation*}
\mathcal{L}_2(U):=\partial_tU+A\partial_{x}U-Q_2U/\epsilon.
\end{equation*}
Clearly, the solution to the original relaxation system satisfies $\mathcal{L}_2(U)=0$. Now we substitute the asymptotic solution into the operator.
By \eqref{IBVP3-v0}---\eqref{IBVP3-u1}, we have 
\begin{align}
\mathcal{L}_2(\bar{U}_0+\epsilon \bar{U}_1)=&\begin{pmatrix}
0\\[1mm]
\bar{E}_1
\end{pmatrix},\quad
\bar{E}_1=\epsilon 
(\partial_{t}\bar{v}_{1} +A_{12}^T\partial_x \bar{u}_1+A_{22}\partial_x \bar{v}_1).
\label{residual-c1}
\end{align}
Based on the regularity assumption, we have $\|\bar{E}_1\|_{L^2(\mathbb{R}^+)}+\|\partial_t\bar{E}_1\|_{L^2(\mathbb{R}^+)}\leq C\epsilon$.
Moreover, we use the equations \eqref{IBVP3-vis-1}---\eqref{IBVP3-vis-4} to derive 
\begin{align}
\mathcal{L}_2(\sqrt{\epsilon}\widehat{U}_0+\epsilon \widehat{U}_1)=&\begin{pmatrix}
0\\[1mm]
\widehat{E}_1
\end{pmatrix}+\widehat{E}_2,\quad
\widehat{E}_1=
\sqrt{\epsilon}(A_{12}^T\partial_z\widehat{u}_1+A_{22}\partial_z\widehat{v}_1),\quad
\widehat{E}_2=\epsilon\partial_{t}\widehat{U}_{1}.\label{residual-c2}
\end{align}
Then we have $\mathcal{L}_2(U_\epsilon)=\mathcal{L}_2(\bar{U}_0+\epsilon \bar{U}_1)+\mathcal{L}_2(\sqrt{\epsilon}\widehat{U}_0+\epsilon \widehat{U}_1)$
and find that $U_{err}=U-U_{\epsilon}$ satisfies the equation in \eqref{estimate-eq1} with the residual term $E_1=\bar{E}_1+\widehat{E}_1$ and $E_2=\widehat{E}_2$. 
Based on the property of solutions to the parabolic differential equations, we know that $\partial_z\widehat{U}_1(z,t)$ and $\partial_t\widehat{U}_1(z,t)$ belong to $L^2(\mathbb{R}^+)$ and thereby
$$
\int_{\mathbb{R}^+}|\widehat{E}_1(\frac{x}{\sqrt{\epsilon}})|^2dx \leq \sqrt{\epsilon}\int_{\mathbb{R}^+}|\widehat{E}_1(z)|^2dz \leq C\epsilon^{3/2},
$$
which means that $\|E_1\|_{L^2(\mathbb{R}^+)}\leq C\epsilon^{3/4}$.
Similarly, we get the estimate $\|\partial_tE_1\|_{L^2(\mathbb{R}^+)}\leq C\epsilon^{3/4}$ and $\|E_2\|_{L^2(\mathbb{R}^+)}+\|\partial_tE_2\|_{L^2(\mathbb{R}^+)}\leq C\epsilon^{5/4}$. Therefore, we verify that $E_1$ and $E_2$ satisfy the condition in Lemma \ref{lemma5.1} with $\kappa=1/4$.

Thanks to assumption (iii) and the uniqueness of the solution to the equation \eqref{IBVP3-ReducedBC}, we infer that the boundary layer terms $\widehat{U}_0$ and $\widehat{U}_1$ should be zero at $t=0$. This shows $U_{err}(x,0)=0$. At last, we recall \eqref{IBVP3-BC} to derive $BU_{err}(0,t)=BU(0,t)-BU_{\epsilon}(0,t)=0$. By Lemma \ref{lemma5.1}, we have 
$$
\max_{t\in[0,T]}\|U_{err}(\cdot,t)\|_{H^1(\mathbb{R}^+)}\leq C\epsilon^{1/4}.
$$

\subsection{Second-type collision term $Q_2$, IBVP 4} 
We assume that 
\begin{itemize}
    \item[(i)] the initial data for the original relaxation problem are given in equilibrium which means 
$$
U(x,0)=\sum_{k=0}^3\epsilon^{k/2} \bar{U}_k(x,0)=\sum_{k=0}^3\epsilon^{k/2} \begin{pmatrix}
   \bar{u}_k\\
   \bar{v}_k
\end{pmatrix}(x,0)
$$
where $\bar{v}_k(x,0)$ satisfies the relation 
\begin{align*}
\bar{v}_0(x,0)=\bar{v}_1(x,0)=0,\quad \bar{v}_2(x,0)=-A_{12}^T\partial_x\bar{u}_0(x,0),\\
\bar{v}_3(x,0)=-A_{12}^T\partial_x\bar{u}_1(x,0).
\end{align*}
\item[(ii)] For $0\leq k\leq 3$, the initial data $\bar{u}_{k}(x,0)$ are smooth enough and compatible with the reduced boundary condition up to a certain order such that $\bar{U}_k\in CH_T^2$ and $\bar{U}_k(0,t)\in H^2(0,T)$ for $0\leq k\leq 3$. 
\item[(iii)] The initial data $U(x,0)$ are compatible with the original boundary condition $B_2U(0,t)=0$ which means $B_2\bar{U}_k(0,0)=0$ for $0\leq k\leq 3$.
\end{itemize}
By the second assumption and the boundary condition \eqref{eq-problem4-ODE}, we know that the boundary value for $\widehat{U}_k(0,t)$ and $\widetilde{U}_k(0,t)$ are also in $H^2(0,T)$. Based on the regularities of parabolic equation and ordinary differential equations, we have $\widehat{U}_k,\widetilde{U}_k\in H^2([0,T]\times\mathbb{R}^+)$. Besides, with a similar argument as previous sections, assumption (iii) implies that the initial data for $\widehat{U}_k(z,0)$ and $\widetilde{U}_k(y,0)$ are zero.
Thus we know that the error term $U_{err}=U-U_\epsilon$ has zero initial data. Furthermore, it is not difficult to see from \eqref{eq-problem4-ODE} that $U_{err}=U-U_\epsilon$ satisfies the boundary condition $B_2U_{err}(0,t)=B_2U(0,t)-B_2U_\epsilon(0,t)=0$.

Next we compute $\mathcal{L}_2(U_\epsilon)=\mathcal{L}_2(\bar{U})+\mathcal{L}_2(\widehat{U})+\mathcal{L}_2(\widetilde{U})$ with $\bar{U}=\sum_{k}\bar{U}_k$, $\widehat{U}=\sum_{k}\widehat{U}_k$ and $\widetilde{U}=\sum_{k}\widetilde{U}_k$. According to \eqref{outer-neq-p-v} and \eqref{outer-neq-p-u}, the outer solution satisfies 
$$
\mathcal{L}_2(\bar{U})=
\epsilon
\begin{pmatrix}
0\\[1mm]
\partial_t\bar{v}_2+A_{12}^T\partial_x\bar{u}_2+ A_{22}\partial_x\bar{v}_2 
\end{pmatrix}
+\epsilon^{3/2}
\begin{pmatrix}
0\\[1mm]
\partial_t\bar{v}_3+A_{12}^T\partial_x\bar{u}_3+ A_{22}\partial_x\bar{v}_3 
\end{pmatrix}
$$
Thanks to the regularity assumption, we write
\begin{align}
\mathcal{L}_2(\bar{U})=&\begin{pmatrix}
0\\
\bar{E}_1
\end{pmatrix}\qquad
\|\bar{E}_1\|_{L^2(\mathbb{R}^+)}+\|\partial_t\bar{E}_1\|_{L^2(\mathbb{R}^+)}\leq C \epsilon.\label{residual-c2-1}
\end{align}
Moreover, for the viscous boundary layer we see from \eqref{vis-k-v} and \eqref{vis-k-u} that
$$
\mathcal{L}_2(\widehat{U})=\epsilon
\begin{pmatrix}
    0\\[1mm]
    \partial_t\widehat{v}_2+A_{12}^T\partial_z\widehat{u}_3 +A_{22}\partial_z\widehat{v}_3
\end{pmatrix}+\epsilon^{3/2}
\partial_t\begin{pmatrix}
    \widehat{u}_3\\[1mm]
    \widehat{v}_3
\end{pmatrix}:=\begin{pmatrix}
    0\\[1mm]
    \widehat{E}_1
\end{pmatrix}+\widehat{E}_2.
$$
Thanks to the regularity assumption, we have 
$$
\|\widehat{E}_1\|_{L^2(\mathbb{R}^+)}^2
=\int_{\mathbb{R}^+}\left|\widehat{E}_1\left(\frac{x}{\sqrt{\epsilon}},t\right)\right|^2dx = \sqrt{\epsilon} \int_{\mathbb{R}^+}\left|\widehat{E}_1\left(z,t\right)\right|^2dz\leq C \epsilon^{5/2}.
$$
It gives the estimate $\|\widehat{E}_1\|_{L^2(\mathbb{R}^+)}\leq C \epsilon^{5/4}$. With the same argument, we also have $\|\partial_t\widehat{E}_1\|_{L^2(\mathbb{R}^+)}\leq C \epsilon^{5/4}$ and $\|\widehat{E}_2\|_{L^2(\mathbb{R}^+)}+\|\partial_t\widehat{E}_2\|_{L^2(\mathbb{R}^+)}\leq C \epsilon^{7/4}$.
At last, we know from \eqref{Knudsen-k} that 
\begin{align}
\mathcal{L}_2(\widetilde{U})=\epsilon\partial_t\widetilde{U}_2+\epsilon^{3/2}\partial_t\widetilde{U}_3:=\widetilde{E}_2.
\label{residual-c2-3}
\end{align}
By the regularity assumption, we have 
$$
\|\widetilde{E}_2\|_{L^2(\mathbb{R}^+)}^2
=\int_{\mathbb{R}^+}\left|\widetilde{E}_2\left(\frac{x}{\epsilon},t\right)\right|^2dx = \epsilon \int_{\mathbb{R}^+}\left|\widetilde{E}_2\left(y,t\right)\right|^2dy\leq C \epsilon^3.
$$
The estimate $\|\partial_t\widetilde{E}_2\|_{L^2(\mathbb{R}^+)}\leq C\epsilon^{3/2}$ can be obtained similarly. To conclude, we take $E_1=\bar{E}_1+\widehat{E}_1$, $E_2=\widehat{E}_2+\widetilde{E}_2$ and verify the conditions in Lemma \ref{lemma5.1} with $\kappa=1/2$. By exploiting the lemma, we obtain
$$
\max_{t\in[0,T]}\|U_{err}(\cdot,t)\|_{H^1(\mathbb{R}^+)}\leq C\epsilon^{1/2}.
$$

\subsection{Conclusion for the second-type collision term $Q_2$} 
We conclude the result for the interface problem with the second-type collision term. Recall that IBVP 3 corresponds to the term $U^{(1)}=\sum_{k=1}^nG^{(k)}$ and IBVP 4 corresponds to the problems for $U^{(i)}=G^{(i)}-G^{(1)}$ with $i\geq 2$. Combining the results above, we obtain
for any $t\in [0,T]$ and $1\leq i\leq n$:
\begin{align*}
\|U^{(i)}(\cdot,t)-U^{(i)}_{\epsilon}(\cdot,t)\|_{L^\infty(\mathbb{R}^+)}\leq C
\|U^{(i)}(\cdot,t)-U^{(i)}_{\epsilon}(\cdot,t)\|_{H^1(\mathbb{R}^+)}\leq C\epsilon^{1/4}.
\end{align*}
Then we define 
\begin{align*}
    G^{(1)}_\epsilon = \frac{1}{n}\Big(U^{(1)}_{\epsilon}-\sum_{k=2}^nU^{(k)}_{\epsilon}\Big),\qquad 
    G^{(i)}_\epsilon = U^{(i)}_\epsilon + G^{(1)}_\epsilon,\quad 2\leq i\leq n.
\end{align*}
and obtain
\begin{align*}
&\|G^{(1)}-G^{(1)}_{\epsilon}\|_{L^\infty(\mathbb{R}^+)} = \frac{1}{n}\|U^{(1)}-U^{(1)}_{\epsilon} - \sum_{k=2}^n\Big(U^{(k)}-U^{(k)}_{\epsilon}\Big)\|_{L^\infty(\mathbb{R}^+)}\leq C\epsilon^{1/4}\\[2mm] 
&\|G^{(i)}-G^{(i)}_{\epsilon}\|_{L^\infty(\mathbb{R}^+)} = \|U^{(i)}-U^{(i)}_{\epsilon} + G^{(1)}-G^{(1)}_{\epsilon} \|_{L^\infty(\mathbb{R}^+)}\leq C\epsilon^{1/4}.
\end{align*}
As $\epsilon$ goes to zero, the asymptotic solution converges to the exact solution in $L^\infty$ norm.

\section{Conclusion}

This paper is a continuation of our preceding works \cite{BDKZ1,BDKZ2} on the derivation
of coupling conditions for macroscopic equations on networks from the underlying kinetic equations and conditions.
The present work rigorously proves the validity of the asymptotic expansions in \cite{BDKZ1,BDKZ2} through an $H^1$ space error estimate. This result establishes the convergence of the discrete kinetic solution towards the macroscopic solution.
Extending the present framework to nonlinear problems remains a challenging topic for future research. Another interesting direction is to extend the proof to cases that include zero kinetic velocity. This issue is related to the theoretical results for general relaxation systems developed in \cite{ZY-chara}, which address the characteristic case.


\end{document}